\documentclass[12pt, a4paper]{amsart}
\usepackage{amsfonts}
\usepackage{amsthm}
\usepackage{amsmath}
\usepackage{longtable}
\usepackage{amssymb}
\usepackage{latexsym,verbatim}
\usepackage{color}
\definecolor{darkblue}{rgb}{0,0,.5}
\usepackage[colorlinks=true,linkcolor=darkblue,urlcolor=violet,citecolor=magenta]{hyperref}
\usepackage{fullpage}

\usepackage{graphs}

\theoremstyle{plain} \newtheorem{lem}{Lemma}[section]
\theoremstyle{plain} \newtheorem{prop}[lem]{Proposition}
\theoremstyle{plain} \newtheorem{thm}[lem]{Theorem}
\theoremstyle{plain} \newtheorem{cor}[lem]{Corollary}

\theoremstyle{definition}
\newtheorem{alg}{Algorithm}
\newtheorem{rem}[lem]{Remark}

\newcommand{\Q}{\mathbb{Q}}

\newcommand{\C}{\mathbb{C}}

\newcommand{\Z}{\mathbb{Z}}
\newcommand{\Aut}{\mathrm{\mathop{Aut}}}
\newcommand{\Hom}{\mathrm{\mathop{Hom}}}
\newcommand{\GL}{\mathrm{\mathop{GL}}}
\newcommand{\SL}{\mathrm{\mathop{SL}}}
\newcommand{\ad}{\mathrm{\mathop{ad}}}
\newcommand{\Lie}{\mathrm{Lie}\,}
\newcommand{\rank}{\mathrm{\mathop{rank}}}

\newcommand{\g}{\mathfrak{g}}
\newcommand{\h}{\mathfrak{h}}
\newcommand{\p}{\mathfrak{p}}
\newcommand{\z}{\mathfrak{z}}

\newcommand{\ssl}{\mathfrak{sl}}

\newcommand{\gt}{\mathfrak}

\numberwithin{equation}{section}

\begin{document}
\hfill {\scriptsize July 10, 2011} 
\vskip1ex

\title[Closures of nilpotent orbits]
{An effective method to compute closure ordering for nilpotent orbits of $\theta$-representations} 
\author{W.A. de Graaf}
\address{Dipartimento di Matematica,
Universit\`{a} di Trento, via Sommarive 14, 
I-38100 Povo (Trento) Italy}
\email{degraaf@science.unitn.it}
\author{E.B. Vinberg}
\address{Chair of High Algebra,
M.V. Lomonosov Moscow State University, 
Vorob'ievy Gory, 119899 Moscow Russia}
\email{vinberg@zebra.ru}
\author{O.S. Yakimova}
\address{Mathematisches Institut, 
Universit\"at Erlangen-N\"urnberg,
Bismarckstrasse 1\,1/2,
91054 Erlangen  Germany}
\email{yakimova@mpim-bonn.mpg.de}
\begin{abstract}
We develop an algorithm for computing the closure of a given nilpotent $G_0$-orbit 
in $\gt  g_1$, where $\gt g_1$ and $G_0$ are coming from a $\Z$ or a $\Z/m\Z$-grading 
$\gt g= \bigoplus \g_i$ of 
a simple complex Lie algebra $\gt g$.  
\end{abstract}
\maketitle

\section{Introduction}

One of the main tasks of 
mathematics is to describe certain objects up to a certain equivalence relation. 
Often this relation is given by an algebraic group action.  
Then equivalence classes are orbits and 
orbit closures correspond to degenerations of our objects. 
Thus, describing orbits of algebraic actions, 
as well as deciding whether one orbit lies in the closure of 
another, is an important and 
interesting problem. However, this is possible only in a
very few cases. One of these instances 
is provided  by the  $\theta$-groups 
introduced by the second author in the seventies, 
see \cite{vinberg}, \cite{vinberg2}. 

Let $G$ be a connected 
reductive complex 
algebraic group and $\g=\Lie G$ its Lie algebra. 
Let $\theta$ be a diagonalisable automorphism of $\g$ that either 
defines a $\Z$ or a $\Z/m\Z$-grading  
$\g =\bigoplus \g_i$,  where the grading components $\g_i$ are 
the eigenspace of $\theta$. 
Note that $\gt g_0=\gt g^\theta$ is the subset of $\theta$-stable points.
Let $G_0\subset G$ be a connected algebraic subgroup such that
$\Lie G_0=\gt g_0$. If $\theta$ extends to an automorphism of $G$, then
$G_0=(G^{\theta})^{\circ}$.
The group $G_0$ is reductive and its natural action on 
$\g_1$ is called a {\it $\theta$-representation}; 
the group $G_0$, together with its action on $\g_1$, is called a {\it $\theta$-group}. 

An orbit $G_0x\subset \g_1$ is said to be {\em semisimple} if it is closed, and
{\em nilpotent} if its closure $\overline{G_0x}$ contains 0. This is the case if and only if $x$ is 
semisimple (respectively nilpotent) as an element of $\g$. 
The elements of $\gt g_1$ inherit the Jordan decomposition 
$x=s+n$ from $\gt g$. Besides, $G_0$-orbits $G_0(s+n)$ with the semisimple part $s$ being fixed up to conjugation are classified by the nilpotent orbits 
of the $\theta$-group coming from the pair $(\gt g_s,\theta|_{\gt g_s})$,
where $\gt g_s\subset\g$ is
the centraliser of $s$ (and it is a reductive Lie algebra) 
and $\theta|_{\gt g_s}$ is the restriction 
of $\theta$ to $\gt g_s$. This indicates that nilpotent orbits 
are  especially interesting . 
The $\theta$-groups have several remarkable properties,
one of them is that there are only finitely many nilpotent $G_0$-orbits in $\gt g_1$ 
and there is a method to classify them \cite{vinberg2}.

From now on suppose that $\gt g$ is simple. 
We will say that a $\theta$-group is exceptional (respectively classical),
if $\g$ is exceptional (respectively classical).
The classical case allows a more or less uniform treatment, since here 
everything is determined by the canonical embedding into an appropriate 
$\gt{gl}_n$, see e.g. \cite{vinberg}. For inner automorphisms 
of $\gt{gl}_n$, the nilpotent orbits as well as their closures are described by  
Kempken (\cite{Gisela}). The complete answer, for all classical types and all automorphisms, is not known, but there does not seem to
be any profound difficulty in getting it. 

More interesting representations arise in the context of exceptional $\theta$-groups.
Here several orbit classifications were carried out along the lines of
\cite{vinberg2}. To mention a few, \cite{elashvin}, \cite{an}, \cite{an-el}, \cite{gatim}. 
In these papers, all orbits, not only the nilpotent ones, were described. 
More recently, Pervushin treated one $\theta$-group in type $E_7$ \cite{per-1}, he also 
got the closure diagram of the nilpotent orbits \cite{per-2}.

Despite the possibility to treat each particular exceptional $\theta$-group by hand,
the ``classical'' uniformity is lost and one faces a long list of different 
examples. Dealing with all of them by hand is at least difficult. 
Several computer algorithms for classifying nilpotent $G_0$-orbits in $\gt g_1$ 
have been developed, see \cite{Peter} and 
\cite{gra15}. 
In this paper, 
we give a method how to check whether a nilpotent orbit $G_0x$ lies in the closure
$\overline{G_0y}$ of another nilpotent orbit $G_0y$.

Each nilpotent element $e\in \g_1$ can be included into an 
$\gt{sl}_2$-triple $(e,h,f)$ with  $h\in \g_0$.  Our method relies 
on the fact that this $h$ is also a characteristic of $e$ in the sense of 
Kempf and Hesselink, i.e., it gives rise to a one-dimensional torus in $G_0$ 
that takes $e$ to zero fastest.   Another important ingredient is that $G_0e$ coincides  
with a Hesselink stratum, the set of all elements in $\gt g_1$ having $h$ as a Hesselink 
characteristic, see \cite[Section~5]{povin}. Therefore 
$\overline{G_0e}=G_0 \left( V_{\geq 2}(h)\right)$, where $V_{\geq 2}(h)$ is the linear span of all
vectors $v\in\gt g_1$ such that $[h,v]=kv$ with $k\geq 2$.

An orbit $G_0e'$ lies in $\overline{G_0e}$ if and only if its  intersection 
with $V_{\geq 2}(h)$ is non-empty.
When examining $G_0e'\cap V_{\geq 2}(h)$, we replace $G_0$ by the 
union of its Bruhat cells.  
Futher, let $h'$ be a characteristic of $e'$ and $W_0$ the Weyl group of $G_0$. Then 
Propositions~\ref{prop:3} assures that $G_0e'$ is contained in $\overline{G_0e}$ if and only 
if there is $w\in W_0$ such that $U(w)=V_2(h')\cap V_{\geq 2}(wh)$ contains a point of $G_0e'$
(here $V_{2}(h')$ is the set of all
vectors $v\in\gt g_1$ such that $[h',v]=2v$). 
If this is indeed the case, then $U(w)\cap G_0e'$ is an open dense subset of $U(w)$ and 
by taking a random $u\in U(w)$ we can find an element of $G_0e'$ with probability almost one. 
In order to prove that the intersection in question is empty, we compute 
the dimension of a maximal $Z(h')$-orbit intersecting $U(w)$ for the centraliser 
$Z(h')\subset G_0$ of $h'$. Recall that $\dim Z(h')v<\dim Z(h')e'$ for all
elements $v$ in $V_2(h')\setminus G_0e'$ (see Lemma~\ref{lem:0}).
To loop over an orbits of 
the Weyl group, we use its parametrisation as a tree with edges given by 
simple reflections (Section~\ref{Weyl-tree-sec}).  
Other tools are described in Sections~\ref{reduced}, \ref{sec:compl}, and 
\ref{sec:alg}. In particular, to prove
a non-inclusion $G_0e'\not\subset\overline{G_0e}$ for some orbits, we use 
Theorem~\ref{char-reduced}, which is a general statement on
$\Z$-graded reductive Lie algebras and is interesting in itself. It already appeared in the
literature and was proved by Kac in a particular case \cite{Kac-some}, 
see Remark~\ref{rem-kac} for a detailed discussion.

\vskip0.5ex

First examples of $\theta$-groups are provided by the simple Lie 
algebras themselves, i.e., in the case where 
the automorphism is the identity. Then
one asks for the Hasse (closure) diagram of the nilpotent orbits in $\g$. 
The two most difficult, largest exceptional Lie algebras, of types $E_7$ and $E_8$, were treated  by 
Mizuno (\cite{mizuno}). Later his results were verified and corrected 
by Beynon and Spaltenstein (\cite{BS84}). 
The implementation of our method in {\sf GAP} also 
works for $\gt g$. We have computed the Hasse diagrams for the Lie algebras of
exceptional type, and obtained the same diagrams as in Spaltenstein's  book 
\cite{spaltenstein}.

The same problem for real exceptional Lie algebras has been studied by Djokovi{\' c} in a series of
papers \cite{dokoclos1,dokoclos7,dokoclos4,dokoclos3,dokoclos2,
dokoclos6,dokoclos8,dokoclos9}. 
If $\g_{\mathbb R}$ is a non-compact real form of $\g$ and $\gt k\subset\gt g_{\mathbb R}$ is the Lie algebra of a maximal compact subgroup in $G_{\mathbb R}$, then 
the complexification $\gt k(\mathbb C)$ of $\gt k$ is a symmetric subalgebra, 
i.e., $\gt k(\mathbb C)=\gt g^\theta$ for $\theta$ of order two. 
The Kostant-Sekiguchi correspondence (see e.g. \cite[\S 9.5]{colmcgov}) establishes  
a bijection between nilpotent $G_{\mathbb R}$ orbits in $\g_{\mathbb R}$ and nilpotent $G_0$-orbits 
in $\g_1$. Moreover, according to \cite{BS}, this bijection preserves the closure ordering. 
For each automorphism of order 2 of each exceptional complex Lie algebra,
Djokovi{\' c} gives the closure diagram for the nilpotent orbits. 
With the implementation of our method in {\sf GAP} we have also computed these diagrams.
The results of our computations were the same as those of Djokovi{\' c}, 
except in one case in type $E_8$. The difference is described in 
Section~\ref{symm}.  

The finite order automorphisms of $\g$ have been classified by Kac (\cite{kac_autom}), 
up to conjugacy. A conjugacy class of automorphisms is identified by its Kac diagram.
Here we briefly indicate how this works for inner automorphisms, for more information we 
refer to \cite[Chapter~3, \S3]{v41} and  \cite[Chapter X]{helgason}. Let $\Phi$  
be the root system of $\g$ with a basis
$\{\alpha_1,\ldots,\alpha_l\}$. 
Let $\alpha_0$ denote the lowest root of $\Phi$. The Dynkin
diagram of the roots $\alpha_0,\alpha_1,\ldots,\alpha_l$ is the extended Dynkin diagram of 
$\Phi$ (or of $\g$). Let $n_i\in \mathbb N$ be such that $\alpha_0 = -\sum_{i=0}^l n_i \alpha_i$ and 
set $n_0=1$. 
Take $l+1$ non-negative integers 
$s_0,\ldots,s_l$ 
with $\gcd(s_0,\ldots,s_l)=1$ and set $m=\sum_{i=0}^l n_is_i$. Let $\omega\in\C$ be a primitive $m$-th root of unity. 
Then a linear map $\theta:\g\to\g$ that multiplies vectors in the root space 
$\g_{\alpha_i}$   ($0\leq i\leq l$)  by  $\omega^{s_i}$
uniquely defines an automorphism of $\g$ of order $m$. 
The Kac diagram of this automorphism (or, more precisely, of its conjugacy class)
is the extended Dynkin diagram with labels $s_0,\ldots,s_l$. The automorphisms that will
appear in the examples in this paper all have the labels $s_i$ equal to $0$ or $1$. We will
give the Kac diagram of such an automorphism by colouring the nodes of the extended Dynkin
diagram: a black node means that the corresponding label is 1, otherwise it is 0.

There is also an easy way to read 
the $\theta$-representation from the Kac diagram of an inner $\theta$.
The group $G_0$ contains a maximal torus of $G$ and the semisimple 
part of $\gt g_0$ is generated by all root spaces $\gt g_{\alpha_i}$  ($0\leq i\leq l$) with $s_i=0$. The lowest weights of $\gt g_1$ (with 
respect to $G_0$) are in one-to-one correspondence 
with the roots labeled with $1$.

There are two instances of $\theta$ groups, one in $E_7$ and one in $E_8$, 
where $G_0$-orbits correspond to isomorphisms classes of two-step nilpotent (or metabelian)  Lie 
 algebras $\gt n$ such that $\gt n'=[\gt n,\gt n]$ is the centre of $\gt n$ and 
either $\dim(\gt n/\gt n')\leq 6,\dim\gt n'\leq 3$; or  $\dim(\gt n/\gt n')\leq 5,\dim\gt n'\leq 5$,
see Section~\ref{sub:2step} and \cite{gatim}. 
The nilpotent orbits correspond to those Lie algebras, whose structure tensor can be contracted to zero by a unimodular change of coordinates.  
Here taking closure 
of a nilpotent orbit can be interpreted as the degeneration of the encoded Lie algebra.
The Lie algebra structures on a given vector space form an affine algebraic variety and 
some of its properties depend on the degenerations, see e.g. \cite{vergne}.
In the Appendix we
present the Hasse diagrams for the nilpotent orbits of both these  $\theta$-representations.  

We have also computed the closures of the nilpotent orbits of $\SL_9(\C)$ 
in $\wedge^3(\C^9)$, see Figures~\ref{fig:3vec1}, \ref{fig:3vec2}. This is a $\theta$-representation treated 
in \cite{elashvin}.

Section~\ref{sec:rm} contains a few further observations on 
algebraic actions. We briefly discuss difficulties arising in developing a practical
algorithm for describing the closure (Section~\ref{sub:ag}); 
outline possible modifications in our algorithm; 
and present a  parametrisation for a set of the double cosets of a 
Weyl group  (Section~\ref{sub:dc}), 
which appeared as a byproduct of our constructions.

\section{Preliminaries}\label{sec:pre}

In this section we present some results, mainly taken from \cite{povin}, on which our method is based. 

Throughout we let $\h_0$ be a fixed Cartan subalgebra of $\g_0$. The Weyl group of
the root system of $\g_0$ with respect to $\h_0$ will be denoted $W_0$. We have
$W_0\cong N_{G_0}(\h_0)/Z_{G_0}(\h_0)$. Hence every  $w\in W_0$ can be lifted
to a $g\in G_0$ such that $g|_{\h_0} = w$. Usually we will denote these two elements
by the same symbol. The group $G$ is assumed to be simple unless explicitly stated 
to the contrary. 

We say that an $\ssl_2$-triple $(h,e,f)$ 
is {\em homogeneous} if $e\in \g_1$, $h\in \g_0$, $f\in \g_{-1}$.

Let us recall a few useful facts (see \cite{vinberg}, \cite{vinberg2}):

\begin{enumerate}
\item For a nilpotent element $e\in \g_1$ there exist $h\in \g_0$, $f\in \g_{-1}$ such that
$(h,e,f)$ is an $\ssl_2$-triple. The element $h$ is called a ({\em Dynkin}) {\em characteristic} of $e$.
\item Let $(h',e',f')$, $(h,e,f)$ be two homogeneous $\ssl_2$-triples. Then
$e'$ and $e$ are $G_0$-conjugate if and only if $(h',e',f')$, $(h,e,f)$ are 
$G_0$-conjugate, if and only if $h'$ and $h$ are $G_0$-conjugate.
\end{enumerate}
Thereby a nilpotent orbit $G_0e$ corresponds to a unique $G_0$-conjugacy class of homogeneous
$\ssl_2$-triples $(h,e,f)$. Also, we may assume that $h$ lies in $\h_0$. Furthermore, 
after possibly replacing $h$ by a $W_0$-conjugate, we may
assume that $h$ lies in a fixed Weyl chamber $C_0$ of $\h_0$. 
Then $h$ is uniquely determined by the orbit $G_0e$.

Throughout we will write $V$ for the space $\g_1$. Then for $h\in \h_0$  we set
$$V_k (h) = \{ v\in V \mid [h,v] = kv \},~~ V_{\geq k}(h) = \bigoplus_{l\geq k} V_l(h).$$
Also we consider the parabolic subalgebra $\p(h)\subset \g_0$, which is the sum of
the eigenspaces of $h$ with non-negative eigenvalues. Let $P(h)$ denote the connected
subgroup of $G_0$ with Lie algebra $\p(h)$. We let $\z(h)$ be the
centraliser of $h$ in $\g_0$. Let $\tilde{\z}(h)$ denote the orthogonal complement
of $h$ in $\z(h)$, with respect to the Killing form of $\g$. Let $Z(h)$ and $\widetilde{Z}(h)$ 
be connected subgroups of $G_0$ with Lie algebras $\z(h)$ and $\tilde{\z}(h)$, respectively.

Now we will borrow two theorems from \cite[Section~5]{povin}.

\begin{thm}[{\cite[Theorem~5.4.]{povin}}]\label{thm:1}
Let $e\in \g_1$ be nilpotent and nonzero. Let $h\in \h_0$ be such that $e \in V_{\geq 2}(h)$.
Then $h$ is a characteristic of $e$ if and only if the projection of $e$ on
$V_2(h)$ is not a nilpotent element with respect to the action of the group
$\widetilde{Z}(h)$.
\end{thm}

\begin{rem}\label{two-char}
In \cite[Section~5]{povin}, the term ``characteristic'' is used in a 
different sense, it is not necessarily a Dynkin characteristic.
However, following the lines of Example~3 in \cite[Section~5.5]{povin},
one can show that the orbit $\tilde Z(h)e$ is closed in $V$, if 
$h$ is a Dynkin characteristic of $e$. Therefore a Dynking characteristic 
of $e$ is also a characteristic in the sense of Theorem~\ref{thm:1}.
\end{rem}

The next theorem is the second part of \cite[Theorem 5.6.]{povin}
and Corollary~\ref{cor:thm2} is an
immediate consequence of Theorem~\ref{thm:2}.

\begin{thm}\label{thm:2}
Let $\mathcal{O}=G_0e$ be a nilpotent orbit in $V$; and let $h\in \g_0$ be a 
characteristic of $e$.
Then $\overline{\mathcal{O}} = G_0 \left( V_{\geq 2}(h)\right)$.
\end{thm}

\begin{cor}\label{cor:thm2}
Let $\mathcal{O}' = G_0e'$, $\mathcal{O}=G_0e$ be two nilpotent orbits in 
$V$. Let $h'$, $h$ be Dynkin characteristics of $e'$, $e$, respectively. 
Then $\mathcal{O}' \subset 
\overline{\mathcal{O}}$ if and only if $V_{\geq 2}(h)$ contains a point of 
$\mathcal{O}'$.
\end{cor}

We use the notation $\g_{i,x}$ for the intersection of $\g_i$ and the centraliser
$\gt g_x\subset\gt g$ of $x\in \g$. Next we have two lemmas that we will use in the sequel.
The first one is an immediate consequence of Theorem~\ref{thm:2}.

\begin{lem}\label{lem:0}
Let $(h,e,f)$ be a homogeneous $\ssl_2$-triple. Then $Z(h)e$ is dense
in $V_2(h)$.
\end{lem}

\begin{lem}\label{lem:1}
Let $(h,e,f)$ be a homogeneous $\ssl_2$-triple, and $\mathcal{O}=G_0e$. Then $h$ is a (Dynkin) characteristic
of all elements of $\mathcal{O}\cap V_2(h)$.
\end{lem}
\begin{proof}
In fact, we are going to prove that 
$Y=\mathcal{O}\cap V_2(h)$ is a single $Z(h)$-orbit, i.e., this intersection is equal 
to $Z(h)e$. Let $y$ be an element of $Y$. 
Let $\mathfrak f_i$ be the eigenspace of $h$ in $\g_0$ with eigenvalue $i$ and
let $\mathfrak f_{i,y}=\mathfrak f_i\cap \g_{0,y}$ 
be the centraliser of $y$ in $\mathfrak f_i$.  
Since $Z(h)e$ is dense in $V_2(h)$, the element $y$ lies in 
its closure. 
In particular, taking the limits one sees that
$\dim\mathfrak f_{i,y}\ge \dim\mathfrak f_{i,e}$ for all $i$. 
On the other hand $\dim\g_{0,e}=\dim\g_{0,y}$, since these are 
the elements of the same $G_0$-orbit. 
Taking into account that $\g_{0,y}=\bigoplus \mathfrak f_{i,y}$
and that $\mathfrak f_{0,y}=\z(h)_y$, 
we conclude that the $Z(h)$-orbits of $e$ and $y$ have the same dimension. 
Since $Z(h)e$ is the unique $Z(h)$-orbit of the maximal dimension,
$y\in Z(h)e$.  

Now the statement about characteristics is obvious. 
\end{proof}

\vskip0.5ex

\subsection{Reduced $\theta$-groups.}\label{reduced} 

We conclude Section~\ref{sec:pre} with a few statements concerning $\theta$-groups 
appearing from $\Z$-gradings.
In this part of the paper, $G$ is an arbitrary 
(not necessarily simple) reductive group.   
A $\Z$-grading of $\gt g=\Lie G$ is defined by 
a diagonalisable one-parameter subgroup of $\Aut(\gt g)$ 
and therefore by the eigenvalues of some $h\in\gt g$, i.e., 
$\gt g_s=\{\xi\in\gt g\mid [h,\xi]=s\xi\}$, see e.g. \cite[Ch.\,3,\,Sec.\,3.3]{v41}. 
Without loss of generality we may assume that
$h\in [\gt g,\gt g]$.
Here all elements of $\gt g_1$ are nilpotent and therefore there is a dense open $G_0$-orbit in $\gt g_1$. 

Let $\rho: G\to \GL(W)$ be a faithful linear representation of $G$ on a finite-dimensional vector space
$W$. 
We use the same letter $\rho$ for its differential 
$\rho:\gt g\to\gt{gl}(W)$ and define
%
a non-degenerate $G$-invariant symmetric scalar product 
$(\,\,,\,)$ on $\gt g$ by setting 
$(x,y):={\rm tr}\,(\rho(x)\rho(y))$ for $x,y\in\gt g$. 
Note that the 
restriction of $(\,\,,\,)$ to each non-abelian 
simple factor of $\gt g$ is the Killing form multiplied by a
positive rational number. One of the benefits of this choice is that $(h,h)>0$,
whenever $h\ne 0$, and this is assumed to be the case.  
More generally, $(s,s)>0$ for all non-zero $s\in[\gt g,\gt g]$ 
that have rational eigenvalues on $\gt g$. 

Let $\tilde{\gt g}_0\subset\gt g_0$ be the orthogonal complement 
of $h$ with respect to $(\,\,,\,)$ and $\tilde G_0\subset G_0$ 
a connected algebraic group with $\Lie\tilde G_0=\tilde{\gt g}_0$.
Then the action of $\tilde G_0$ on $\gt g_1$ is said to be 
a {\it reduced} $\theta$-representation and $\tilde G_0$ a {\it reduced}
$\theta$-group. Note that $G_0=\tilde G_0(\exp(\C h))$.   

\begin{lem}\label{conic}
Let $x\in\gt g_1$. Then $\tilde G_0x=G_0x$ if and only if 
$[\tilde{\gt g}_0,x]=[\gt g_0,x]$, and the equality takes place if and 
only if the orbit $\tilde G_0x$ is conical.  
\end{lem}
\begin{proof}
If $\tilde G_0x=G_0x$, then clearly $[\tilde{\gt g}_0,x]=[\gt g_0,x]$. 
Other way around, the equality of tangent spaces implies that 
$\dim G_0x=\dim\tilde G_0x$. Since $\tilde G_0$ is a normal subgroup of $G_0$, the same holds for all elements in $G_0x$ and the two orbits coincide. 
Finally, being conical means that $\C^{^\times}\!\!x\subset \tilde G_0x$, or, equivalently, 
$\C x\subset [\tilde{\gt g}_0,x]$. Therefore  $\tilde G_0x$ is a conical orbit if and only if  
there is the equality of orbits or their tangent spaces. 
%
\end{proof}

\begin{lem}\label{D-char}
Suppose that $2h$ is a Dynkin characteristic of $x\in\gt g_1$. Then 
$[\gt g_0,x]=\gt g_1$, but $[\gt{\tilde g}_0,x]\ne \gt g_1$. 
\end{lem}
\begin{proof}
If $[\gt g_0,x]\ne \gt g_1$, then there is $v\in\gt g_{-1}$ such that 
$([\gt g_0,x],v)=0$ and also $(\gt g_0,[x,v])$, where $[x,v]\in\gt g_0$.
Since the scalar product is non-degenerate on $\gt g_0$, we obtain $[x,v]=0$, 
which contradicts the $\gt{sl}_2$-theory. 

There is an element $y\in\gt g_{-1}$ such that $y$, $2h$, and $x$ form an 
$\gt{sl}_2$-triple. For this $y$ we have 
$(y,[\tilde{\gt g}_0,x])=0$, because $(2h,\tilde{\gt g}_0)=0$.
The inequality follows. 
\end{proof}

\begin{thm}\label{char-reduced} 
Let $G$ be an arbitrary reductive group and the objects 
$\gt g_1$, $G_0$, $\tilde G_0$ as above.  
Suppose that $x\in\gt g_1$. 
Then $\tilde G_0x\ne G_0x$ if and only if 
$2h$ is a Dynkin characteristic of $x$. 
\end{thm}
\begin{proof}
Let $\hat h\in\gt g_0$ be a Dynkin characteristic of $x$. 
We can write it 
as $\hat h=ah+h_0$ with $a\in\C$ and $h_0\in\tilde{\gt g}_0$. 
Since $[\hat h,x]=2x=ax+[h_0,x]$,  either $\C x\subset [\tilde{\gt g}_0,x]$ or
$a=2$ and $[h_0,x]=0$. 
In the latter case $(\hat h,h_0)=0$. 
Taking into account the equality $(h,h_0)=0$, we get that $(h_0,h_0)=0$.
Since $\hat h$ is a Dynkin characteristic, it lies in $[\gt g,\gt g]$. 
Hence $h_0\in[\gt g,\gt g]$, because  $h$ also does.
Moreover, eigenvalues of $\ad(h)$ are integers by the construction, 
and the same holds for $\ad(\hat h)$, because it comes from an $\gt{sl}_2$-triple.
Since $[h,\hat h]=0$, the eigenvalues of $\ad(h_0)$ are integers as well. 
According to our choice of the scalar product,
the equality $(h_0,h_0)=0$ is possible only if $h_0=0$. 
One concludes that $2h$ is a Dynkin characteristic of $x$.

We have shown that if $2h$ is not a Dynkin characteristic of $x$, 
then $[h_0,x]=bx$ with $b\in\C^{^\times}$, in particular, $\tilde G_0x$ is a conical orbit. 
By Lemma~\ref{conic}, $G_0x=\tilde G_0x$.

If $2h$ is a Dynkin characteristic of $x$,
then $[\tilde{\gt g}_0,x]\ne [\gt g_0,x]$ by Lemma~\ref{D-char} and therefore 
$G_0x\ne\tilde G_0x$.
%
%
\end{proof}

\begin{rem}\label{rem-kac}
In case $\gt g$ is simple and the representation of $G_0$ on
$\gt g_1$ is irreducible, Theorem~\ref{char-reduced} was proved by 
V.\,Kac, see \cite[Proposition~3.2]{Kac-some}. It is also mentioned 
without a proof in \cite[Section~8.5]{povin} that the statement holds 
for an arbitrary reduced $\theta$-group. Since we could not find 
a general case proof in the literature, we decided to include it here. 
\end{rem}

\section{Criteria for inclusion}

In this section we state and prove the main criterion (Proposition \ref{prop:3})
that we use for deciding
whether a given nilpotent orbit is contained in the closure of another given
nilpotent orbit. This reduces the problem of checking inclusion to a finite number of
checks, each corresponding to an element of a certain orbit of the Weyl group $W_0$.
Subsequently we give some observations that help when using the criterion.

\begin{prop}\label{prop:3}
Let the notation be as in Corollary~\ref{cor:thm2}. Then $\mathcal{O}' \subset 
\overline{\mathcal{O}}$ if and only if there is a $w\in W_0$ such that 
$U=V_2(h') \cap V_{\geq 2}(wh)$ contains a point of $\mathcal{O}'$.
Moreover, in that case the intersection of $U$ and $\mathcal{O}'$ is dense in $U$.
\end{prop}
\begin{proof}
The ``if'' part follows directly from Theorem~\ref{thm:2}. Therefore suppose that 
$\mathcal{O}' \subset \overline{\mathcal{O}}$.

By the Bruhat decomposition we have that
$$ G_0 = \bigcup_{w\in W_0} P(h')wP(h).$$
By Theorem~\ref{thm:2}, 
\begin{align*}
\overline{\mathcal{O}} &= G_0\left(V_{\geq 2}(h)\right) \\
&= \bigcup_{w\in W_0} P(h') wP(h) \left(V_{\geq 2}(h)\right)\\
&= \bigcup_{w\in W_0} P(h') w \left(V_{\geq 2}(h)\right).
\end{align*}
Let $(h',e',f')$ be a homogeneous $\ssl_2$-triple. 
Then it follows from the above that there exist $w\in W_0$, $p\in P(h')$, 
and $x\in V_{\geq 2}(h)$ with $e' = pwx$, or, equivalently, $p^{-1}e' = wx$. 

Next $P(h') = Z(h')\ltimes N$, where $N$ is a connected subgroup of $G_0$ 
whose Lie algebra is the sum of the eigenspaces of $h'$ in $\g_0$ 
with positive eigenvalues. In particular, 
$p^{-1} = ln$ with $l\in Z(h')$ and  $n\in N$. 
Since $e'\in V_2(h')$, we have $ne' = e' + y$, 
where $y\in V_{\geq 3}(h')$. Now $p^{-1}e' = le' + ly$, with $le' \in V_2(h')$ and
$ly \in V_{\geq 3}(h')$. 
In particular, 
$p^{-1}e'$ lies in $V_{\geq 2}(h')$.
Since 
$p^{-1}e' =wx$  and $wx\in V_{\geq 2}(wh)$, it also lies in 
$\tilde U= V_{\geq 2}(h') \cap V_{\geq 2}(wh)$. 

The elements $h'$ and $wh$ commute and thereby $\tilde U$ is stable 
under the action of $h'$. That is, $\tilde U$ is
the direct sum of $h'$-eigenspaces. It follows that $\tilde U$ contains $le'$, which is 
obviously an element of $\mathcal{O}'$. 
Moreover, 
$le'\in V_{2}(h')$ and hence $le'\in U$, where 
$U=V_{2}(h')\cap V_{\geq 2}(wh)$. 

By Theorem \ref{thm:1}, an element $v\in U$ lies in $\mathcal{O}'$ if and only if it
is not nilpotent with respect to the action of $\widetilde{Z}(h')$. Threfore if the intersection
of $U$ and $\mathcal{O}'$ is not empty, then it has to be open and dense. 
\end{proof}

\begin{prop}\label{prop:4}
Let $(h',e',f')$, $(h,e,f)$ be homogeneous $\ssl_2$-triples, with $e' \in 
\mathcal{O}'$, $e \in \mathcal{O}$. 
Let $\kappa$ denote the Killing form of $\g$. If $\kappa(h',h) < 
\kappa(h',h')$ then $V_2(h')\cap V_{\geq 2}(h)$ contains no points of 
$\mathcal{O}'$.
\end{prop}
\begin{proof}
Note that $h \in \z(h')$, hence $h = a h' +t$, where $a\in \C$ and $t\in \tilde\z(h')$.
Moreover, 
$$a = \frac{\kappa(h',h)}{\kappa(h',h')},$$
which is in $\Q$ and $<1$. Hence $t$ has only positive eigenvalues on 
$V_2(h')\cap V_{\geq 2}(h)$.  Let $T$ be the connected subgroup of $G_0$
whose Lie algebra is spanned by $t$. Then all elements of $V_2(h')\cap V_{\geq 2}(h)$
are nilpotent with respect to $T$, and in particular with respect to $\widetilde{Z}(h')$.
Hence by Theorem \ref{thm:1} and Lemma \ref{lem:1}, 
the former space contains no points of $\mathcal{O}'$.
\end{proof}

Let $\mathfrak{l}$ be a Lie algebra acting on a vector space $M$. Then for $v\in M$ we
denote its stabiliser by $\mathfrak{l}_v$, i.e., 
$$\mathfrak{l}_v = \{ x\in \mathfrak{l} \mid x{\cdot} v = 0\}.$$
The set of $v\in M$ with $\dim \mathfrak{l}_v$ minimal is open and dense in $M$. 

\begin{prop}\label{prop:5}
Let $(h',e',f')$, $(h,e,f)$ be homogeneous $\ssl_2$-triples, with $e' \in 
\mathcal{O}'$, $e \in \mathcal{O}$.
Let $d = \dim \z(h')_{e'}$. Let $d'$ be the minimal dimension of 
$\z(h')_u$, for $u\in V_2(h')\cap V_{\geq 2}(h)$. Then $d\leq d'$. Moreover,
$V_2(h')\cap V_{\geq 2}(h)$ contains a point of $\mathcal{O}'$ if and 
only if $d=d'$
\end{prop}
\begin{proof}
By Lemma \ref{lem:0}, the stabiliser of $e'$ in $\z(h')$ has minimal
possibile dimension. 

Furthermore, if $d=d'$ then there is $u\in V_2(h')\cap V_{\geq 2}(h)$ such that
$\dim \z(h')_u = \dim \z(h')_{e'}$. Hence the dimension of the $Z(h')$-orbit of $u$ is
the same as the dimension of $Z(h')e'$. So $Z(h')u$ is dense in $V_2(h')$ as well.
The conclusion is that $Z(h')e' = Z(h')u$, and $u$ lies in $\mathcal{O}'$.
\end{proof}

\begin{prop}\label{prop:6}
Let $(h',e',f')$, $(h,e,f)$ be homogeneous $\ssl_2$-triples, with $e \in 
\mathcal{O}$, $e' \in \mathcal{O}'$. Set $U = V_2(h')\cap V_{\geq 2}(h)$. Let $\mathfrak{n} = 
N_{\g_0}(U) = \{ x \in \g_0 \mid [x,U]\subset U\}$. Let $u\in U$; if $[\mathfrak{n},u]=
U$, and $u\not\in \mathcal{O}'$, then $U$ has no point of $\mathcal{O}'$.
\end{prop}
\begin{proof}
Indeed, if $U$ has a point of $\mathcal{O}'$, then the intersection of $\mathcal{O}'$ and
$U$ is dense in $U$. But also the $N_{G_0}(U)$-orbit of $u$ is dense in $U$. So the two sets
must intersect, which is not possible.
\end{proof}

\section{Orbits of the Weyl group}\label{Weyl-tree-sec}

In our algorithm we need to loop over an orbit $W_0h$, where $h\in \h_0$. In this 
section we briefly describe how this is done. For simplicity we assume that the centre
of $\g_0$ is zero. If this is not the case then $\g_0$ has to be replaced by its
derived subalgebra $[\g_0,\g_0]$, and $\h_0$ by its intersection with $[\g_0,\g_0]$.

We let $\kappa$ denote the Killing form of $\g$. Since it is non-degenerate on $\h_0$ it
gives an isomorphism $\h_0^*\to \h_0$, $\alpha \mapsto \hat{\alpha}$. This yields an
inner product on $\h_0^*$ by $(\alpha,\beta) = \kappa( \hat{\alpha}, \hat{\beta} )$.

Let $\Phi_0$ be the root system of $\g_0$ with respect to $\h_0$. Let 
$\Delta_0 =\{\alpha_1,\ldots,\alpha_l\}$ be a fixed basis of $\Phi_0$. The corresponding set of 
positive roots will be denoted $\Phi_0^+$.

For $\alpha\in \Phi_0$ we set
$$\alpha^\vee = \frac{2 \hat{\alpha}}{(\alpha,\alpha)}\in \h_0.$$
The Weyl group $W_0$ is generated by the simple reflections $s_i=s_{\alpha_i}$. For 
$h\in \h_0$ we have $s_i(h) = h -\alpha_i(h)\alpha_i^\vee$. 

We use a basis $h_1,\ldots,h_l$ of $\h_0$, defined by $\alpha_i(h_j) = \delta_{ij}$.
Then, if $h = \sum_i a_i h_i$, we get $s_j(h) = h - a_j\alpha_j^\vee$. The elements $h$
of which we compute the $W_0$-orbit,
lie in an $\ssl_2$-triple. This implies that the coefficients of $h$
with respect to this basis are integers. 
The dominant Weyl chamber $C_0$ consists 
of the elements of $\h_0$ having non-negative coefficients with respect to the basis
$h_1,\ldots,h_l$. 

Now let $h\in \h_0$ be the element of which we want to compute the orbit $W_0h$. Since
every orbit of $W_0$ has a unique point in $C_0$, we may assume that $h\in C_0$. 
Let $\hat h\in W_0h$, then we define the length of $\hat h$, denoted $\ell(\hat h)$, 
as the length of
a shortest $w\in W_0$ with $\hat h=wh$. Then
$$\ell(\hat h) = |\{ \alpha\in \Phi_0^+ \mid \alpha(\hat h) <0 \}|.$$
This implies that $\ell(s_i\hat h) = \ell(\hat h)+1$ if and only if $a_i >0$, where 
$\hat h = \sum_i a_i h_i$. We use a criterion due to Snow (\cite{snow1}):

\begin{lem}\label{Weyl-tree}
Let $\tilde h=\sum_i a_i h_i$ be an element of $W_0h$ of length $k+1$. Then there is a unique
$\hat h$ of length $k$ in $W_0h$ such that
\begin{itemize}
\item there is a simple reflection $s_i$ with $s_i(\hat h)=\tilde h$,
\item $a_j\geq 0$ for $i < j\leq l$.
\end{itemize}
\end{lem}

Let $\tilde h, \hat h$ be as in the previous lemma. Then we say that $\hat h$ is the 
{\em predecessor} of $\tilde h$, and conversely, that $\tilde h$ is a {\em successor} of $\hat h$.
Let $\hat h = \sum_i b_i h_i$ be a given element of $W_0h$ of length $k$. 
Then it is straightforward
to determine its successors. Indeed, let $i$ be such that $b_i > 0$, and write 
$s_i(\hat h) = \sum_i a_i h_i$. Then this element is of length $k+1$, and it is a successor
of $\hat h$ if and only if $a_j \geq 0$ for $i < j \leq l$.

This means that we can define a tree: the nodes are the elements of $W_0h$, and there
is an edge from $\hat h$ to $\tilde h$ if and only if $\tilde h$ is a successor of $\hat h$.
By traversing this tree, we can efficiently
loop over $W_0h$. Every element of $W_0h$ comes at the cost of applying one reflection.
Moreover, we do not obtain the same element of $W_0h$ twice. 

\begin{rem}\label{rem:kdescent}
We finish this section with an observation that will be used later.
Let $h'$ be an element of $C_0$.
Let $\hat h \in W_0h$ be of length $k$ 
and suppose that $s_j \hat h$ is of length $k+1$. Write $\hat h = \sum_i a_i h_i$; then, 
as seen above, $a_j > 0$. Hence
$$\kappa(s_j\hat h, h') = \kappa( \hat h, h')-a_j \kappa(\alpha_j^\vee,h') \leq 
\kappa(\hat h, h').$$
Furthermore, equality happens if and only if $\alpha_j(h')=0$, which is equivalent
to $s_j$ lying in the stabiliser of $h'$.
\end{rem}

\section{Complement of the dense orbit}\label{sec:compl}

According to Proposition~\ref{prop:3}, we will have to check whether a 
subspace $U\subset V_2(h)$ contains a point of the dense orbit 
$Z(h)e$. If $U$ contains a point of $Z(h)e$ then the intersection of $U$ and 
$Z(h)e$ is dense in $U$ (Proposition \ref{prop:3}). So in that case, 
by trying random elements of $U$, we quickly find a $u\in U$ lying in $Z(h)e$; thus
proving that the intersection is non-empty. 
The most difficult part of the problem is to prove that 
$U$ contains no points of $Z(h)e$. Here we present two possible solutions. 

Let $v_1,\ldots,v_s$ and $x_1,\ldots,x_n$ be bases of 
$V_2(h)$ and $\z(h)$ respectively. Let also $w_1,\ldots, w_s$ (with $w_i\in V_2(h)^*$)
be the dual basis. Let $B$ denote the action matrix for the representation of 
$\z(h)$ on $V_2(h)^*$.  To be more explicit, the entries of $B$ are elements 
of $V_2(h)^*$, $b_{ij}=x_i{\cdot}w_j$. For $v\in V$, let $B_v$ denote the 
restriction of $B$ to $v$. The entries of this new matrix are 
$[x_i{\cdot}w_j](v)=w_j([v,x_i])$.
In the same spirit, we can define  the restriction of $B$ to $U$, $B_U$,
to be a matrix with entries in $U^*$. The rank of $B_U$ is calculated over 
the field $\mathbb C(U)$ (note that $U^*\subset \mathbb C(U)$). 

Using the fact that $[\xi,v]=0$ (with $\xi\in\z(h)$) if and only if 
$w_i([\xi,v])=0$ for all $i$, one can easily deduce that 
\begin{equation}\label{gen-rank}
\begin{array}{l}
({\sf i})\  \dim\z(h)_v=n-\rank B_v \ \text{for all $v\in V_2(h)$}; \\
({\sf ii})\  \dim Z(h)v=\rank B_v; \\
({\sf iii})\  \max\limits_{u\in U}\dim Z(h)u=\rank B_U; \\
({\sf iv})\ U\cap Z(h)e\ne\varnothing \ \text{if and only if $\rank B_U=s$.}  
\end{array}
\end{equation}

Depending on $s$ and $n$, computing the rank of $B_U$
over a function field may turn out to be rather time consuming. 
For this reason we also consider an alternative method, based
on another characterisation of the elements in $V_2(h)\setminus Z(h)e$,
which comes from Theorem~\ref{char-reduced}.

\begin{prop}\label{split}
Take $v\in V_2(h)$. 
Then the three conditions: $Z(h)v=\tilde Z(h)v$, $[\z(h),v]=[\tilde\z(h),v]$, and 
$v\in V_2(h)\setminus Z(h)e$, are equivalent.  
\end{prop}
\begin{proof}
We are going to identify $\tilde Z(h)$ with a reduced $\theta$-group. To this end, 
for each $i\in\Z$, set $\hat i:=i \mod m$, if $\theta$ has a finite order $m$; and 
$\hat i:=i$ otherwise. Then $\gt l=\bigoplus\limits_{i\in\mathbb Z}\gt l_i$, 
where $\gt l_i=(\gt g_{\hat i})_{2i}(h)$, is a $\Z$-graded 
Lie subalgebra of $\gt g$ with $\gt l_1=V_2(h)$ and 
$\gt l_0=\gt z(h)$. Let $L\subset G$ be a connected subgroup with $\Lie L=\gt l$.
Since $\kappa$ defines a non-degenerate pairing 
between $(\gt g_{\hat i})_{2i}(h)$ and $(\gt g_{\hat{j}})_{-2i}(h)$ with $j=-i$, 
we get a non-degenerate $L$-invariant scalar product 
$(\,\,,\,):=\kappa|_{\gt l}$
on $\gt l$. In particular, $\gt l$ is a reductive subalgebra. 
Here $e\in(\gt g_1)_{2}(h)=\gt l_1$, $f\in(\gt g_{-1})_{-2}(h)=\gt l_{-1}$ and 
therefore $h\in[\gt l,\gt l]$. Note that the $\Z$-grading on $\gt l$ 
is defined by the eigenvalues of $h/2$.  

Recall that  $\gt g$ is assumed to be simple. 
Restricting the adgoint action of $G$ to $L$ we get a faithful representation
$\rho$ of $L$ on $\gt g$ such that $(x,y)={\rm tr}\,(\rho(x)\rho(y))$ 
for $x,y\in\gt l$ and $(\,\,,\,)=\kappa|_{\gt l}$. 
Thus, we are in the setting of Section~\ref{reduced} and can apply  
Theorem~\ref{char-reduced} to  
the $\Z$-graded reductive Lie algebra $\gt l$. 
Here $\tilde L_0=\tilde Z(h)$ and $\gt l_1=(\gt g_1)_{2}(h)=V_2(h)$.  
 
We have $[\z(h),v]=[\tilde\z(h),v]$ if and only if the $\tilde Z(h)$-orbit 
$\tilde Z(h)v$ is conical. Besides, $h$ is a Dynkin characteristic of
all elements in $Z(h)e$. Therefore both equivalences follow from 
Theorem~\ref{char-reduced}.   
\end{proof}

Assume that the basis $x_1,\ldots,x_n$ is chosen in such a way that 
$x_1=h$ and $x_2,\ldots,x_n$ form a basis of $\tilde\z(h)$. Let $\tilde B$ be a submatrix 
of $B$ consisting of the last $n-1$ rows (corresponding to the Lie subalgebra 
$\tilde\z(h)$). Let also $\tilde B_U$ be the restriction of $\tilde B$ to $U$. 
Since $\dim Z(h)v=\rank B_v$ and $\dim\tilde Z(h)v=\rank\tilde B_v$, 
Proposition~\ref{split}  gives us the following:
\begin{equation}\label{eq:split}
U\cap Z(h)e=\varnothing \ \text{if and only if $\rank B_U=\rank\tilde B_U$.}
\end{equation}
In other words, either $\rank B_u=\rank\tilde B_u$ or 
$u$ is an element of $Z(h)e$ and $\rank B_u=s$. 
The equality in equation~(\ref{eq:split}) is satisfied if and only if the first row of $B_U$ 
lies in the linear span  of the rows of $\tilde B_U$. In order to check this we use the following 
steps.

\begin{enumerate}
\item Take a random $u\in U$, compute the rank of $\tilde B_u$, say 
$\rank\tilde B_u=r$. 
\item Find an $r\times r$ non-zero minor of $\tilde B_u$, without loss of generality 
suppose that it is given by the first $r$ rows and the first $r$ columns. 
\item Check whether the first row of $B_U$ is contained in the span of the first
$r$ rows of $\tilde B_U$.
\end{enumerate}

If $u\in U$ is generic, i.e.,  
$\rank\tilde B_u=\rank\tilde B_U$, then the first $r$ rows of $\tilde B_U$ span
the row space of $\tilde B_U$. Hence step~(3) verifies whether 
the first row of $B_U$ is contained in the row space of $\tilde B_U$. Moreover, this will be
the case if and only if $U\cap Z(h)e$ is empty. (Also note that the check in the third
step can be done by computing $s{-}r$ minors of size $r{+}1$.)
Even if $u$ is not a generic element, it may still be true that the first row of 
$B_U$ is contained in the span of the first $r$ rows of $\tilde B_U$,
and the above procedure will prove that $U\cap Z(h)e=\varnothing$.

In many cases it is easier to carry out this procedure than to check
the inequality $\rank B_U<s$. For example, some 
$32\times 38$-matrices $B_U$ appeared while checking non-inclusions 
for a half-spin representation of $D_8$ (line 3 in Table~\ref{tab:time}) 
and 2760681 minors would have to be computed for them.
In other cases it may be easier to deal with the whole matrix, if,
for example, $B_U$ contains a zero column.  

It is not obvious beforehand which choice is the best.
In the implementation of our algorithm we do the following: if $s-n < s-r$, then 
it is checked whether $\rank B_U<s$. Otherwise we check whether the first
row if $B_U$ is contained in the first $r$ rows of $\tilde B_U$, using the procedure
outlined above. We do not claim that this always gives the best choice,
but {\it some} choice is better than none.

If it turns out that the first row of $B_U$ is not contained in the span
of the first $r$ rows of $\tilde B_U$, then it may still be the case that the intersection
is empty (if this happens, then necessarily  $\rank\tilde B_U>\rank\tilde B_u$). 
Then we will have to compute the rank of $B_U$. 
However, the probability of this event can be made arbitrarily small.

\section{The main algorithm}\label{sec:alg}

Here we describe our algorithm for deciding whether one of the two given nilpotent $G_0$-orbits in $\g_1$ lies in the closure of the other. 

First we consider the following problem:
given a homogeneous $\ssl_2$-triple
$(h,e,f)$ and $e'\in V_2(h)$, decide whether $e' \in G_0e$. We have a straightforward
solution for that, based on Lemma~\ref{lem:1}. The existence of $f'\in \g_{-1}$ with
$[h,f']=-2f'$ and $[e',f']=h$ is equivalent to a system of linear equations. We solve this 
system; if it has a solution then $e'$ lies in $G_0e$, otherwise it does not.

Throughout we fix a basis of the root system of $\g_0$ with respect to $\h_0$. 
Then the Weyl group $W_0$ is
generated by the reflections corresponding to the elements of this basis. Furthermore,
this choice also fixes a dominant Weyl chamber $C_0\in \h_0$. 
As before we let $\kappa$ denote the Killing form
of $\g$.

\begin{alg}
Input: two homogeneous $\ssl_2$-triples, $(h',e',f')$, $(h,e,f)$, such that $h',h\in C_0$.\\
Output: {\sc true} if $\mathcal{O}'=G_0e'$ is contained in the closure of 
$\mathcal{O}=G_0e$, {\sc false} otherwise.
\begin{enumerate}
\item If $\kappa(h',h) < \kappa(h',h')$ then return {\sc false}. 
Else go to the next step.
\item For all elements $wh\in W_0h$ do the following:
\begin{enumerate}
\item  If $\kappa(h',wh) \geq \kappa(h',h')$ then:
\begin{enumerate}
\item Select a random $u\in U=V_2(h')\cap V_{\geq 2}(wh)$. 
\item If $u\in \mathcal{O}'$ then return {\sc true}. Otherwise go to the next step.
\item Set $\mathfrak{n} = N_{\g_0}(U)$. If $[\mathfrak{n},u] \neq U$ then 
decide whether $U\cap Z(h)e$ is empty using the methods of Section~\ref{sec:compl}.
If the intersection is not empty then return {\sc true}.
\end{enumerate}
\end{enumerate}
\item If in the previous loop {\sc true} was never returned, then return {\sc false}.
\end{enumerate}
\end{alg}

\begin{prop}
The previous algorithm terminates correctly.
\end{prop}
\begin{proof}
It is obvious that the algorithm terminates, we must show that the output is correct.
We claim that the algorithm checks whether there is a $wh \in W_0 h$ such that
$U(wh)=V_2(h') \cap V_{\geq 2}(wh)$ contains a point of $\mathcal{O}'$. Then
by Proposition \ref{prop:3} the output is correct.

First of all we note that since $h',h\in C_0$ we have that the maximal value
of $\kappa(h',wh)$, for $wh\in Wh$, is $\kappa(h',h)$ (see Remark \ref{rem:kdescent}). 
Therefore, if $\kappa(h',h) < 
\kappa(h',h')$ then no space $U(wh)$ contains a point of $\mathcal{O}'$ (Proposition
\ref{prop:4}). So in this case we are immediately done.

Otherwise we inspect every $wh\in W_0h$. If $\kappa(h',wh) < \kappa(h',h')$ then
$U(wh)$ contains no points of $\mathcal{O}'$ by Proposition~\ref{prop:4}. 
So then we can discard it. 
Otherwise we select a random $u\in U(wh)$. If $u\in \mathcal{O}'$, then we are done.
If not, and $[\mathfrak{n},u]=U(wh)$, then $U(wh)$ has no points of $\mathcal{O}'$ by
Proposition~\ref{prop:6}. Finally $U(wh)$ contains an element with a minimal dimensional 
stabiliser if and only if $U(wh)$
has a point of $\mathcal{O}'$ by Proposition~\ref{prop:5}. 
\end{proof}

\begin{rem}
Note that if $U(wh)$ contains a point of $\mathcal{O}'$, 
then the set of such points
is open and dense in $U(wh)$, by Proposition~\ref{prop:3}. Hence in that case the random choice
has a high probability of finding an element $u\in \mathcal{O}'$.
\end{rem}

\begin{rem}
Now we make some observations that help to execute the algorithm more efficiently.
\begin{itemize}
\item Of course, we apply the algorithm only if $\dim \mathcal{O}' < \dim \mathcal{O}$, 
as otherwise there is no inclusion.
\item Inclusion also implies that $\dim \g_{k,e'} \geq \dim \g_{k,e}$ for all $k$; so if that
condition is not fulfilled, we also do not apply the algorithm. 
\item When looping over the orbit $W_0h$ we use the tree structure described in the
previous section. When doing this, several shortcuts can be made. First of all
if $\kappa(h',wh) < \kappa(h',h')$ then the entire subtree below $wh$ can be discarded,
by Remark \ref{rem:kdescent}.
Second, if $V_2(h')\cap V_{\geq 2}(wh)$ contains no points of $\mathcal{O}'$, and
the successor of $wh$ is $s_iwh$, where $s_i$ lies in the stabiliser of $h'$, then
also $V_2(h')\cap V_{\geq 2}(s_iwh)$ contains no points of $\mathcal{O}'$. 
Hence in that case
we can immediately jump to the next element of the orbit.
\item We collect the subspaces $U(wh)=V_2(h') \cap V_{\geq 2}(wh)$ that appear during the
execution of the algorithm. If a certain such subspace is contained in one that was
treated before, then we already know that it contains no points of $\mathcal{O}'$.
So in that case we can immediately go to the next round. 
All  calculations are done in a
basis $v_1,\ldots,v_s$ of $V_2(h')$ consisting of $\gt h_0$-eigenvectors. 
Each subspace $U(wh)$ is a linear span of $v_i$ such that
$(wh){\cdot}v_i=a_iv_i$ with $a_i\geq 2$. 
Therefore storing and verifying inclusions among the $U(wh)$ is a binary 
problem. 
\end{itemize}
\end{rem}

We have implemented this algorithm in the language of the computer algebra system
{\sf GAP}4 (\cite{gap4}), on top of the {\sf SLA} package (\cite{sla}), which has
implementations of algorithms to list the nilpotent orbits of a $\theta$-group.
One of the main problems for the practical computation lies in the methods of
Section \ref{sec:compl}, where minors of a matrix with entries in a function field
have to be computed. For these computations we use 
the computer algebra system {\sc Magma} (\cite{magma}). We have chosen
this system, because it has very efficient implementations of algorithms to compute
the determinant of a matrix with entries in a function field.

In Table~\ref{tab:time} we collect some experimental data with respect to the implementation
of our algorithm. All computations have been performed on a 3.16 GHz machine. 

\setlength{\unitlength}{1pt}
\begin{table}[htb]
\begin{tabular}{|r|l|r|r|r|}
\hline
$|\theta|$ & Kac diagram of $\theta$ & $\#$ orbits & {\sf GAP} & {\sc Magma} \\
\hline
1 & 
\begin{picture}(150,40)
  \put(3,4){\circle{6}}
  \put(23,4){\circle{6}}
  \put(43,4){\circle{6}}
  \put(63,4){\circle{6}}
  \put(83,4){\circle{6}}
  \put(43,24){\circle{6}}
  \put(6,4){\line(1,0){14}}
  \put(26,4){\line(1,0){14}}
  \put(46,4){\line(1,0){14}}
  \put(66,4){\line(1,0){14}}
  \put(43,7){\line(0,1){14}}
\put(103,4){\circle{6}}
\put(86,4){\line(1,0){14}}
\put(123,4){\circle{6}}
\put(106,4){\line(1,0){14}}
\put(143,4){\circle*{6}}
\put(126,4){\line(1,0){14}}
\end{picture} 
& 69 & 1003 & 397\\
2 & 
\begin{picture}(150,40)
  \put(3,4){\circle{6}}
  \put(23,4){\circle{6}}
  \put(43,4){\circle{6}}
  \put(63,4){\circle{6}}
  \put(83,4){\circle{6}}
  \put(43,24){\circle{6}}
  \put(6,4){\line(1,0){14}}
  \put(26,4){\line(1,0){14}}
  \put(46,4){\line(1,0){14}}
  \put(66,4){\line(1,0){14}}
  \put(43,7){\line(0,1){14}}
\put(103,4){\circle{6}}
\put(86,4){\line(1,0){14}}
\put(123,4){\circle*{6}}
\put(106,4){\line(1,0){14}}
\put(143,4){\circle{6}}
\put(126,4){\line(1,0){14}}
\end{picture} 
& 36 & 124 & 0\\
2 &
\begin{picture}(150,40)
  \put(3,4){\circle*{6}}
  \put(23,4){\circle{6}}
  \put(43,4){\circle{6}}
  \put(63,4){\circle{6}}
  \put(83,4){\circle{6}}
  \put(43,24){\circle{6}}
  \put(6,4){\line(1,0){14}}
  \put(26,4){\line(1,0){14}}
  \put(46,4){\line(1,0){14}}
  \put(66,4){\line(1,0){14}}
  \put(43,7){\line(0,1){14}}
\put(103,4){\circle{6}}
\put(86,4){\line(1,0){14}}
\put(123,4){\circle{6}}
\put(106,4){\line(1,0){14}}
\put(143,4){\circle{6}}
\put(126,4){\line(1,0){14}}
\end{picture} 
& 115 & 900 & 0.18 \\
3 & 
\begin{picture}(150,40)
  \put(3,4){\circle{6}}
  \put(23,4){\circle{6}}
  \put(43,4){\circle{6}}
  \put(63,4){\circle{6}}
  \put(83,4){\circle{6}}
  \put(43,24){\circle*{6}}
  \put(6,4){\line(1,0){14}}
  \put(26,4){\line(1,0){14}}
  \put(46,4){\line(1,0){14}}
  \put(66,4){\line(1,0){14}}
  \put(43,7){\line(0,1){14}}
\put(103,4){\circle{6}}
\put(86,4){\line(1,0){14}}
\put(123,4){\circle{6}}
\put(106,4){\line(1,0){14}}
\put(143,4){\circle{6}}
\put(126,4){\line(1,0){14}}
\end{picture}  
& 101 & 444 & 0.09\\
5 & 
\begin{picture}(150,40)
  \put(3,4){\circle{6}}
  \put(23,4){\circle{6}}
  \put(43,4){\circle{6}}
  \put(63,4){\circle*{6}}
  \put(83,4){\circle{6}}
  \put(43,24){\circle{6}}
  \put(6,4){\line(1,0){14}}
  \put(26,4){\line(1,0){14}}
  \put(46,4){\line(1,0){14}}
  \put(66,4){\line(1,0){14}}
  \put(43,7){\line(0,1){14}}
\put(103,4){\circle{6}}
\put(86,4){\line(1,0){14}}
\put(123,4){\circle{6}}
\put(106,4){\line(1,0){14}}
\put(143,4){\circle{6}}
\put(126,4){\line(1,0){14}}
\end{picture}  
& 105 & 335 & 0.04 \\[0.2ex]
\hline 
\end{tabular} \\[0.5ex]
\caption{Running times (in seconds) for the algorithm applied to several automorphisms
of the Lie algebra of type $E_8$. The first column lists the order of $\theta$, and the second
column its Kac diagram. The third column has the number of nilpotent orbits. The fourth and
fifth columns display, respectively, the time needed for executing the {\sf GAP} part
of the program, and the time spent in the {\sc Magma} part.}\label{tab:time}
\end{table}

From the table we see that the {\sf GAP} part essentially has no problems, also with
large examples. On some occasions it is not necessary to execute the {\sc Magma} part,
as with the second automorphism in the table. On other occasions this part has a trivial
running time, as with the last three examples. However, it also happens that a fair amount
of time is spent in the {\sc Magma} part, as with the first example.

\section{Further remarks on groups and orbits}\label{sec:rm} 

Here we collect some theoretical observations that could 
have been used in the algorithm, but turned out not to be 
of much practical value. 

Let $(f,h,e)$ be a homogeneous $\gt{sl}_2$-triple as before, we also keep 
all the previous notation, including $V_2(h)$. First we consider the actions 
of $Z(h)$ and $\tilde Z(h)$  on $V_2(h)$ more closely. 
As was already mentioned, $Z(h)$ acts on $V_2(h)$ with a dense 
open orbit, $Z(h)e$.

\subsection{The semi-invariant $P$.}\label{sub:inv} 

The stabiliser $\gt z(h)_e$, being the centraliser of an 
$\gt{sl}_2$-subalgebra generated by $e$, $h$, and $f$,  
is reductive and therefore the orbit $Z(h)e$ is an affine space.
This implies that the complement  
$V_2(h)\setminus Z(h)e$ is a divisor
and is a zero-set of 
a single semi-invariant polynomial, say $P$. 
To check whether a subspace $U=U(wh')=V_2(h)\cap V_{\ge 2}(wh')$ 
intersects the dense orbit, one just has
to look on the restriction of $P$ to $U$. 
In this terms, ${\mathcal O}$ lies in the closure 
of ${\mathcal O}'$ if and only if there is 
$U(wh')$ such that $P$ is non-zero on it. 
This could be a replacement for both: choosing a random element 
in $U$ and generic rank considerations.   

One can try to compute a polynomial $P$ by hand. This may involve typing errors and 
time-consuming calculations. 
It is also possible to get $P$ from the matrix 
$B$ with entries $b_{ij}=x_i{\cdot}w_j$, 
where $\{x_i\}$ is a basis of $\gt z(h)$ and $\{w_j\}$ is a basis of 
$V_2(h)^*$ (this matrix was already considered  in Section~\ref{sec:compl}.)    
The polynomial $P$ is the greatest common divisor 
of the largest, $\dim V_2(h){\times}\dim V_2(h)$, minors of $B$. 
In some cases the resulting formula is rather bulky and 
not easy to deal with, in some other {\sc Magma} was unable to finish the calculation. 
It turns out that {\sc Magma} checks much more easily that 
the restriction $B_U$ of $B$ to $U$ does not have the maximal rank, 
$\dim V_2(h)$, than it computes the greatest common divisor of minors. 
Thus
we gave up the idea of using $P$.

\subsection{Double cosets of Weyl groups.}\label{sub:dc} 

Suppose that we have two characteristics 
$h$ and $h'$ lying in the dominant chamber of $W_0$.
Parametrisation of $W_0h$ involves a certain numbering of 
simple roots $\alpha_1,\ldots,\alpha_l$, see Section~\ref{Weyl-tree-sec}.    
This numbering can be arbitrary. 
The stabiliser $W_{0,h'}$ is a Weyl subgroup generated 
by $s_i$ with $\alpha_i(h')=0$. 
Assume that $\rank W_{0,h'}=r$ and the simple roots orthogonal to 
$h'$ have numbers from $l{-}r{+}1$ to $l$. 

\begin{lem}\label{Weyl-double}
Keep the above notation and enumeration of simple roots.
Let ${\bf T}$ be a tree parametrising the orbit $W_0h$,
constructed according to the principles of Lemma~\ref{Weyl-tree}. 
Then the nodes $h$ and $s_i\hat h$ of ${\bf T}$ with $i\le l{-}r$ 
are in one-to-one correspondence with the double cosets 
$W_{0,h'}\backslash W_0/W_{0,h}$.   
\end{lem}
\begin{proof}
First note that $h$ or $s_i\hat h$ with $i\le l{-}r$
lies in the dominant chamber of $W_{0,h'}$. 
Secondly, an element $s_i\hat h$ with $l{-}r< i$ does not bring a new double 
coset, because here $s_i$ lies in $W_{0,h'}$. 
\end{proof}

This, of course, is not a very effective way for listing the double cosets
as the whole orbit $W_0h$ has to be constructed. 
However, if some time consuming calculation has to be performed for 
representatives of double cosets, such a treatment may be useful. 

In our situation, collecting subspaces $U(wh)$ turned out to be much
more effective than refining the Weyl-group tree. The explanation is that
one and the same $U=U(wh)$ arises for many different elements $wh\in W_0h$. 

\subsection{Other algebro-geometric methods.}\label{sub:ag}  
Our algorithm is designed for $\theta$-groups and works quite well.
There are known some other, more general, approaches, which unfortunately
have a rather small range of application.

To begin with, consider a linear action  $a: Q\times V\to V$
of an affine algebraic group $Q$ on a vector 
space $V$. For $x\in V$, the map $g\mapsto gx$ from $Q$ to $V$ is regular, i.e., given by polynomials. 
Let $I(Qx)\lhd\C[V]$ be the ideal of $Qx$, i.e., a set of all
polinomials vanishing on $Qx$.
In \cite{clo}, algorithmic methods are described for computing 
generators of the vanishing ideal 
of the image of a regular map. In particular, this can be applied to $I(Qx)$.
Here $I(Qx)$ equals $(a^*)^{-1}(I(Q{\times}\{x\}))$, where
$I(Q{\times}\{x\})\lhd\C[\GL(V){\times}\{x\}]$ is the defining ideal of 
the product of the image of $Q$ in $\GL(V)$ and the point $x$.
Once the generators of $I(Qx)$ are known, it is  straightforward to decide whether a point 
(and hence the orbit of that point) lies in $\overline{Qx}$.

In order to use the algorithms for getting generators of $I(Qx)$, we need 
as input the polynomials defining $Q$ as a subgroup of $\GL(V)$. 
In our setting, $V=\g_1$ and
$Q$ is the image in $\GL(V)$ of $G_0$, acting on $V$. In order to get equations for
$Q$, methods from \cite{graeqns} can be used. However, both the algorithm for 
obtaining the polynomials defining $Q$ and the one for
computing $I(Qx)$ 
heavily rely on Gr\"obner basis computations. These are extremely time consuming. 
For this reason this method is only applicable to very small examples (e.g., when the
semisimple part of $\g_0$ is of type $A_1$, and $\g_1$ is of dimension 5).

Now let $G\subset \GL[V]$ be reductive.  Then in the above considerations
$G$ can be replaced by its big open cell, $BwB$, where $w\in W$ is the longest 
element in the Weyl group $W$ of $G$ and $B\subset G$ is a Borel subgroup.  
In \cite{po}, V.L.\,Popov suggested an algorithm, based on this observation, 
for deciding whether $Gy$ lies in $\overline{Gx}$. That 
algorithm uses a system of linear equations in $n\binom{n+2d-3}{n-1}$
variables, where $n=\dim V$ and $d$ is the degree of $G$ as 
a subvariety of ${\rm L}(V)$.
For an irreducible $5$-dimensional representation of $\SL_2(\C)$, the number of 
variables equals 56794400. 
This makes it difficult, if not impossible, to use the
algorithm for practical
computations.

\section{Examples}\label{ex}

In this section we show the output of our programs on several examples. 
Here we describe the examples; the next section contains the Hasse diagrams
that we computed with the algorithm, as well as tables giving the characteristics
of the nilpotent orbits.

\subsection{Symmetric pairs}\label{symm}

As was mentioned in the Introduction, the order two case, or, in other words,  the symmetric case, was studied by Djokovi{\' c}, because of its relationship with simple real Lie algebras. We have checked all symmetric pairs arising from the 
exceptional Lie algebras. The result is that Djokovi{\' c} diagrams are basically correct, 
if one takes into account the necessary alteration that he found himself,
\cite{dokoclos6}, \cite{dokoclos9}. Our calculations confirm these corrections.
Apart from this, there are two inclusions missing 
for one automorphism in type $E_8$. For the involution in question, $\g_0$ is of type 
$D_8$ and $\g_1$ is a half-spin representation, the corresponding Kac diagram 
is the third one in Table~\ref{tab:time}. 

According to our calculations, in Table~2 of \cite{dokoclos8}, $59 \to 53$,
and $95 \to 92$ should be added (notation as in the mentioned paper: $59 \to 53$ means that
orbit number $53$ is contained in the closure of orbit number $59$). 
This then results in several other changes. For example,  
\cite[Table 2]{dokoclos8} states ``$99 \to 92, 94, 95$'' and $92$
has to be removed from here, because the orbit number $92$ does not give 
rise to an irreducible component of the boundary 
$\partial{\mathcal O}_{99}$. 
 
Neither $(59,53)$ nor $(95,92)$ appears in the list 
\cite[Table~5]{dokoclos8} of the critical pairs, i.e.,  pairs
$(a,b)$ such that the non-inclusion $a\not\to b$ is (or has to be) proved.  
In both cases, our program immediately found that the space 
$U=V_2(h')\cap V_{\ge 2}(h)$ contains a point of ${\mathcal O}'$.\footnote{This can be even verified by hand, would someone wish to do so.} 

\subsection{Trivectors of a nine dimensional space}

In \cite{elashvin}, the orbits of $\SL_9(\C)$ acting on $\wedge^3(\C^9)$ were
obtained. This is known as the classification of the trivectors of a nine dimensional
space. The orbits were obtained by realising this representation as a $\theta$-representation.
Here $\theta$ is an automorphism of order $3$ of the Lie algebra of type $E_8$, with
Kac diagram

%

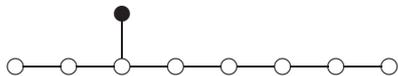
\setlength{\unitlength}{1pt}
\begin{picture}(150,40)

  \put(3,10){\circle{6}}
  \put(23,10){\circle{6}}
  \put(43,10){\circle{6}}
  \put(63,10){\circle{6}}
  \put(83,10){\circle{6}}
  \put(43,30){\circle*{6}}
  \put(6,10){\line(1,0){14}}
  \put(26,10){\line(1,0){14}}
  \put(46,10){\line(1,0){14}}
  \put(66,10){\line(1,0){14}}
  \put(43,13){\line(0,1){14}}
\put(103,10){\circle{6}}
\put(86,10){\line(1,0){14}}
\put(123,10){\circle{6}}
\put(106,10){\line(1,0){14}}
\put(143,10){\circle{6}}
\put(126,10){\line(1,0){14}}
\end{picture} 

\noindent
We have that $\g_0\cong \ssl_9(\C)$ and $\g_1\cong \wedge^3(\C^9)$.  

In Table \ref{tab:3vecchars} we list the characteristics of the (nonzero) nilpotent
orbits. A characteristic $h$ is given by the values of $\alpha_i(h)$, where 
$\{\alpha_1,\ldots,\alpha_8\}$ is a basis of the root system of $\g_0$. Moreover, all the
characteristics $h$ lie in the dominant Weyl chamber with respect to this basis.

Figures \ref{fig:3vec1}, \ref{fig:3vec2} contain, respectively, the top half
and the bottom half of the Hasse diagram. 

\subsection{The classification of metabelian Lie algebras}\label{sub:2step}

A finite-dimensional Lie algebra $L$ is said to be {\it metabelian} 
(or two-step nilpotent), 
if $[L,[L,L]]=0$.
In \cite{gatim}, Galitski and Timashev 
described the $G_0$-orbits for the two particular $\theta$-representations 
in order to obtain the classification of the metabelian Lie algebras of dimensions up to 9 (over algebraically closed fields of characteristic 0). 
With our algorithm we have 
computed the closure diagram of the nilpotent orbits in both these cases.

For the first $\theta$-group, $\theta$ is
an automorphism of order $5$ of the Lie algebra of type $E_8$, with Kac diagram

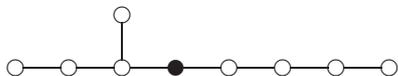
\setlength{\unitlength}{1pt}
\begin{picture}(150,40)

  \put(3,10){\circle{6}}
  \put(23,10){\circle{6}}
  \put(43,10){\circle{6}}
  \put(63,10){\circle*{6}}
  \put(83,10){\circle{6}}
  \put(43,30){\circle{6}}
  \put(6,10){\line(1,0){14}}
  \put(26,10){\line(1,0){14}}
  \put(46,10){\line(1,0){14}}
  \put(66,10){\line(1,0){14}}
  \put(43,13){\line(0,1){14}}
\put(103,10){\circle{6}}
\put(86,10){\line(1,0){14}}
\put(123,10){\circle{6}}
\put(106,10){\line(1,0){14}}
\put(143,10){\circle{6}}
\put(126,10){\line(1,0){14}}
\end{picture} 

\noindent
Characteristics of the nilpotent orbits are given in Table \ref{tab:5ordchars}. 
The closure diagram is displayed in Figures \ref{fig:5ord1},
\ref{fig:5ord2}.

For the second $\theta$-group, $\theta$ is the
automorphism of order $3$ of the Lie algebra of type $E_7$, with Kac diagram

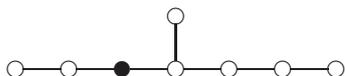
\setlength{\unitlength}{1pt}
\begin{picture}(150,40)

  \put(3,10){\circle{6}}
  \put(23,10){\circle{6}}
  \put(43,10){\circle*{6}}
  \put(63,10){\circle{6}}
  \put(83,10){\circle{6}}
  \put(63,30){\circle{6}}
  \put(6,10){\line(1,0){14}}
  \put(26,10){\line(1,0){14}}
  \put(46,10){\line(1,0){14}}
  \put(66,10){\line(1,0){14}}
  \put(63,13){\line(0,1){14}}
\put(103,10){\circle{6}}
\put(86,10){\line(1,0){14}}
\put(123,10){\circle{6}}
\put(106,10){\line(1,0){14}}
\end{picture} 

\noindent
Characteristics of 
the nilpotent orbits are given in Table~\ref{tab:charsE7}. 
The closure diagram is displayed in Figures~\ref{fig:E7_3},
\ref{fig:E7_3-2}.

As was already mentioned, in \cite{gatim} the 
orbits of these particular
$\theta$-groups are used for a classification of metabelian Lie algebras.
Every orbit corresponds to one such Lie algebra (up to isomorphism).
Every metabelian Lie algebra $L$ has a signature, that is a pair $(m,n)$
where $m= \dim L/[L,L]$ and $n= \dim [L,L]$. 
Let $Z\lhd L$ be a maximal Abelian ideal such that
$Z\cap [L,L]=0$. Then $L\cong L/Z\oplus Z$. 
In the closure diagrams we indicate the signature of $L/Z$ as it was 
computed in \cite{gatim}. 
Mostly this is done by writing the label of the node in a particular font, 
according to Tables~\ref{tab:fonts1} (for Figures~\ref{fig:5ord1}, \ref{fig:5ord2})
and \ref{tab:fonts2} (for Figures~\ref{fig:E7_3}, \ref{fig:E7_3-2}). 
For the orbits corresponding to signatures not present in these tables we have
put the signature in the diagram, next to the node.

\begin{table}[htb]
\begin{tabular}{|l|c|c|}
\hline
font & example & signature\\
\hline
roman & 10 & (5,5)\\
bold face & {\bf 14} & (5,4)\\
italics & {\it{\color{blue}70}} & (5,3) \\
underline & {\underline{82}} & (4,4)\\
typewriter &  {\tt 92} & (4,3)\\
\hline
\end{tabular}
\\[.5ex]
\caption{Fonts for the Hasse diagram in Figures \ref{fig:5ord1},
\ref{fig:5ord2}.}\label{tab:fonts1}
\end{table}

\begin{table}[htb]
\begin{tabular}{|l|c|c|}
\hline
font & example & signature\\
\hline
roman & 10 & (6,3)\\
bold face & {\bf 35} & (6,2)\\
italics & {\it{\color{blue}41}} & (5,3) \\
typewriter &  {\tt 68} & (4,3)\\
overline & $\overline{64}$ & (4,2)\\
\hline
\end{tabular}
\\[.5ex]
\caption{Fonts for the Hasse diagram in Figures \ref{fig:E7_3},
\ref{fig:E7_3-2}.}\label{tab:fonts2}
\end{table}

Taking the closure of a given nilpotent 
orbit corresponds to the degeneration of the encoded 
two-step nilpotent Lie algebra. Let us explain this.    
Let $W$ be a vector space (over $\C$). Then a Lie bracket on $W$ can be 
seen as an element of $\Hom( \wedge^2 W, W)$. 
The group $\GL(W)$ acts on $\Hom( \wedge^2 W, W)$,
and the orbits of this action are in one-to-one correspondence with the isomorphism classes of Lie algebra structures on $W$.
Let $\lambda,\mu \in \Hom( \wedge^2 W, W)$;
if $\mu$ is contained in the closure of the $\GL(W)$-orbit of $\lambda$, then $\mu$
is said to be a {\em degeneration} of $\lambda$. We refer to \cite{gruno} for an
introduction into this concept. 

In relation to the variety of metabelian Lie algebras one
considers two vector spaces $U$ and $V$. A metabelian Lie bracket on $U\oplus V$ is viewed
as an element of $\Hom( \wedge^2 U, V )$.  Let $L$ be the Lie algebra defined by
such an element; then $[L,L] \subset V$. The group $\GL(U){\times}\GL(V)$ acts on 
$\Hom( \wedge^2 U, V )$.  Two metabelian Lie algebra structures on $U\oplus V$ are isomorphic 
if and only if the corresponding elements of $\Hom( \wedge^2 U, V )$ lie in the same 
$\GL(U){\times}\GL(V)$-orbit. 

Write $m=\dim U$, $n=\dim V$, and let $W$ be a vector space of dimension $m+n$.
If $\lambda \in \Hom( \wedge^2 W, W)$ is a metabelian Lie bracket, defining a Lie algebra
$L$ on $W$ of a signature $(m,n)$, then by setting $V=[L,L]$ and taking $U$ to be a complement of
$V$ in $W$, we get an element of $\Hom( \wedge^2 U, V )$. This construction preserves
isomorphism, i.e., a $\GL(W)$-orbit of metabelian Lie brackets in 
$\Hom( \wedge^2 U, V )$
is mapped to a $\GL(U){\times}\GL(V)$-orbit in $\Hom( \wedge^2 U, V )$. 
Thus metabelian Lie algebras of signature $(m,n)$ are classified 
by the $\GL(U){\times}\GL(V)$-orbits in $\Hom(\wedge^2 U, V)$.
Moreover, degenerations of these Lie algebras are given by the 
orbits closures.

Two instances of $\theta$-groups considered in \cite{gatim} correspond to
the signatures $(5,5)$ and $(6,3)$. Strictly speaking,
the group $G_0$ is semisimple in both cases. Therefore, if $0\notin\overline{G_0x}$
with $x\in\gt g_1$, then the one-parameter family 
of $G_0$-orbits $G_0(ax)$ with $a\in\C^{^\times}$ gives only one isomorphism class 
of metabelian Lie algebras. For the nilpotent $G_0$-orbits, there is no 
difference between $G_0x$ and $(\GL(U){\times}\GL(V))x$.

Note that the quotients $L/Z$ with signatures being smaller than
or equal to $(5,3)$ appear in both $\theta$-groups. Therefore, the lower parts of both
Hasse diagrams (Figures~\ref{fig:5ord1}, \ref{fig:5ord2}; \ref{fig:E7_3}, \ref{fig:E7_3-2})
are the same.     

The affine  variety of metabelian Lie algebras structures with 
signature $(5,5)$ on a ten-dimensional vector space $W$  is irreducible, because 
of the equivalence with $\theta$-group orbits. According to \cite{gatim},
it has a one-parameter family of the maximal $\GL(W)$-orbits. 
The same holds for the second case, where signatures are $(6,3)$ and $\dim W=9$,
only there is a two-parameter family of the maximal orbits \cite{gatim}.
Since isomorphism classes of  metabelian Lie algebras with smaller signatures 
are parametrised by nilpotent orbits, there are only finitely many of them.  
In addition,
our diagrams show that for each signature $(m,n)$, where either
$m\leq 5,\,n\leq 5$ and $(m,n)\ne (5,5)$, or 
$m\leq 6,\,n\leq 3$ and $(m,n)\ne (6,3)$,
the affine  variety of metabelian Lie algebra structures with 
signature $(m,n)$ on  $W$ with $\dim W=m+n$ is irreducible.  
For example, all metabelian  Lie algebras with signature $(5,4)$ 
are degenerations of $L/Z$ (including $L/Z$ itself), where $L$ corresponds to the orbit 
number 14 of the first $\theta$-group (in $E_8$), up to isomorphism, see Figure~\ref{fig:5ord1}. 
The metabelian  Lie algebras with signature $(5,3)$ 
are degenerations of $L/Z$, where $L$ is encoded by the orbit number 62 of the first $\theta$-group as well as by  the orbit number 27 of the second $\theta$-group, see Figures~\ref{fig:5ord1}, \ref{fig:E7_3}.


\def\cprime{$'$} \def\cprime{$'$} \def\Dbar{\leavevmode\lower.6ex\hbox to
  0pt{\hskip-.23ex \accent"16\hss}D} \def\cprime{$'$} \def\cprime{$'$}
  \def\cprime{$'$} \def\cprime{$'$}


\section*{Appendix: diagrams and tables}

\unitlength=1cm
\begin{figure}

\begin{graph}(10,21)
\roundnode{n1}(4.5,20)\nodetext{n1}{1}
\roundnode{n2}(4.5,19)\nodetext{n2}{2}
\roundnode{n3}(4,18)\nodetext{n3}{3}
\roundnode{n4}(5,18)\nodetext{n4}{4}
\roundnode{n5}(4,17)\nodetext{n5}{5}
\roundnode{n6}(5,17)\nodetext{n6}{6}
\roundnode{n7}(5,16)\nodetext{n7}{7}
\roundnode{n8}(4,16)\nodetext{n8}{8}
\roundnode{n9}(3,16)\nodetext{n9}{9}
\roundnode{n10}(6,15)\nodetext{n10}{10}
\roundnode{n11}(4,15)\nodetext{n11}{11}
\roundnode{n12}(5,15)\nodetext{n12}{12}
\roundnode{n13}(3,15)\nodetext{n13}{13}
\roundnode{n14}(5,14)\nodetext{n14}{14}
\roundnode{n15}(4,14)\nodetext{n15}{15}
\roundnode{n16}(3,14)\nodetext{n16}{16}
\roundnode{n17}(7,13)\nodetext{n17}{17}
\roundnode{n18}(6,13)\nodetext{n18}{18}
\roundnode{n19}(4,13)\nodetext{n19}{19}
\roundnode{n20}(5,13)\nodetext{n20}{20}
\roundnode{n24}(7,12)\nodetext{n24}{24}
\roundnode{n21}(3,12)\nodetext{n21}{21}
\roundnode{n22}(5,12)\nodetext{n22}{22}
\roundnode{n23}(6,12)\nodetext{n23}{23}
\roundnode{n25}(4,12)\nodetext{n25}{25}
\roundnode{n26}(4,11)\nodetext{n26}{26}
\roundnode{n27}(7,11)\nodetext{n27}{27}
\roundnode{n28}(5,11)\nodetext{n28}{28}
\roundnode{n29}(3,11)\nodetext{n29}{29}
\roundnode{n30}(6,11)\nodetext{n30}{30}
\roundnode{n31}(4,10)\nodetext{n31}{31}
\roundnode{n32}(5,10)\nodetext{n32}{32}
\roundnode{n33}(6,10)\nodetext{n33}{33}
\roundnode{n34}(7,10)\nodetext{n34}{34}

\roundnode{n35}(4,9)\nodetext{n35}{35}
\roundnode{n36}(5,9)\nodetext{n36}{36}
\roundnode{n37}(6,9)\nodetext{n37}{37}
\roundnode{n38}(4,8)\nodetext{n38}{38}
\roundnode{n39}(5,8)\nodetext{n39}{39}
\roundnode{n40}(6,8)\nodetext{n40}{40}

\roundnode{n41}(3,7)\nodetext{n41}{41}
\roundnode{n42}(8,7)\nodetext{n42}{42}
\roundnode{n43}(4,7)\nodetext{n43}{43}
\roundnode{n44}(6,7)\nodetext{n44}{44}
\roundnode{n45}(5,7)\nodetext{n45}{45}

\roundnode{n46}(5,6)\nodetext{n46}{46}
\roundnode{n47}(3,6)\nodetext{n47}{47}
\roundnode{n48}(4,6)\nodetext{n48}{48}
\roundnode{n49}(7,6)\nodetext{n49}{49}
\roundnode{n50}(6,6)\nodetext{n50}{50}

\roundnode{n51}(6,5)\nodetext{n51}{51}
\roundnode{n52}(5,5)\nodetext{n52}{52}
\roundnode{n53}(4,5)\nodetext{n53}{53}

\roundnode{n54}(3,4)\nodetext{n54}{54}
\roundnode{n55}(4,4)\nodetext{n55}{55}
\roundnode{n56}(5,4)\nodetext{n56}{56}
\roundnode{n57}(7,4)\nodetext{n57}{57}

\roundnode{n58}(3,3)\nodetext{n58}{58}
\roundnode{n59}(6.8,3)\nodetext{n59}{59}
\roundnode{n60}(4,3)\nodetext{n60}{60}
\roundnode{n61}(5,3)\nodetext{n61}{61}
\roundnode{n62}(7.5,3)\nodetext{n62}{62}

\roundnode{n63}(6,2)\nodetext{n63}{63}
\roundnode{n64}(5,2)\nodetext{n64}{64}

\roundnode{n65}(3,1)\nodetext{n65}{65}
\roundnode{n66}(6,1)\nodetext{n66}{66}
\roundnode{n67}(4,1)\nodetext{n67}{67}
\roundnode{n68}(7,1)\nodetext{n68}{68}

\roundnode{n69}(5,0)\nodetext{n69}{69}
\roundnode{n70}(7,0)\nodetext{n70}{70}

\edge{n1}{n2}
\edge{n2}{n3}
\edge{n2}{n4}
\edge{n3}{n5}
\edge{n3}{n6}
\edge{n4}{n5}
\edge{n4}{n6}
\edge{n5}{n9}
\edge{n5}{n8}
\edge{n5}{n7}
\edge{n6}{n8}
\edge{n6}{n7}
\edge{n9}{n13}
\edge{n8}{n13}
\edge{n8}{n11}
\edge{n8}{n12}
\edge{n7}{n12}
\edge{n7}{n10}
\edge{n7}{n11}

\edge{n13}{n16}
\edge{n13}{n15}
\edge{n11}{n16}
\edge{n11}{n14}
\edge{n12}{n15}
\edge{n12}{n14}
\edge{n12}{n17}
\edge{n10}{n17}
\edge{n10}{n18}
\edge{n10}{n20}
\edge{n16}{n20}
\edge{n16}{n19}
\edge{n15}{n19}
\edge{n15}{n24}
\edge{n14}{n19}
\edge{n14}{n18}

\edge{n19}{n21}
\edge{n19}{n25}
\edge{n19}{n22}
\edge{n20}{n22}
\edge{n18}{n22}
\edge{n18}{n23}
\edge{n17}{n22}
\edge{n17}{n23}
\edge{n17}{n24}
\edge{n20}{n26}
\edge{n29}{n33}

\edge{n21}{n29}
\edge{n21}{n28}
\edge{n25}{n29}
\edge{n25}{n30}
\edge{n22}{n27}
\edge{n22}{n28}
\edge{n22}{n30}
\edge{n24}{n28}
\edge{n24}{n30}
\edge{n23}{n27}
\edge{n23}{n30}

\edge{n26}{n31}
\edge{n28}{n32}
\edge{n28}{n33}
\edge{n28}{n34}
\edge{n30}{n32}
\edge{n30}{n33}
\edge{n27}{n31}
\edge{n27}{n32}
\edge{n27}{n33}
\edge{n27}{n34}

\edge{n31}{n35}
\edge{n31}{n41}
\edge{n32}{n35}
\edge{n32}{n36}
\edge{n32}{n37}
\edge{n33}{n36}
\edge{n33}{n37}
\edge{n34}{n37}
\edge{n35}{n47}
\edge{n35}{n38}
\edge{n35}{n39}
\edge{n36}{n38}
\edge{n36}{n40}
\edge{n36}{n42}
\edge{n36}{n44}
\edge{n37}{n39}
\edge{n37}{n40}

\edge{n38}{n43}
\edge{n38}{n46}
\edge{n39}{n43}
\edge{n39}{n45}
\edge{n40}{n43}
\edge{n40}{n45}

\edge{n41}{n47}
\edge{n41}{n46}
\edge{n43}{n48}
\edge{n43}{n52}
\edge{n45}{n48}
\edge{n45}{n55}
\edge{n45}{n50}
\edge{n44}{n50}
\edge{n44}{n57}

\edge{n42}{n49}
\edge{n42}{n46}
\edge{n40}{n49}

\edge{n46}{n52}
\edge{n47}{n52}
\edge{n47}{n55}
\edge{n48}{n53}
\edge{n48}{n54}
\edge{n49}{n51}
\edge{n49}{n52}
\edge{n50}{n51}
\edge{n50}{n53}
\edge{n51}{n56}
\edge{n52}{n56}
\edge{n53}{n56}

\edge{n54}{n58}
\edge{n54}{n60}
\edge{n55}{n60}
\edge{n56}{n60}
\edge{n55}{n61}
\edge{n51}{n61}
\edge{n51}{n59}
\edge{n53}{n59}
\edge{n57}{n62}
\edge{n53}{n62}

\edge{n60}{n64}
\edge{n61}{n64}
\edge{n59}{n64}
\edge{n56}{n63}
\edge{n59}{n63}
\edge{n58}{n65}
\edge{n60}{n65}
\edge{n61}{n67}

\edge{n64}{n66}
\edge{n63}{n66}
\edge{n63}{n68}
\edge{n62}{n68}
\edge{n67}{n69}
\edge{n66}{n69}
\edge{n66}{n70}
\edge{n68}{n70}

\freetext(0,20.5){Dimension}
\freetext(0,20){80}
\freetext(0,19){79}
\freetext(0,18){78}
\freetext(0,17){77}
\freetext(0,16){76}
\freetext(0,15){75}
\freetext(0,14){74}
\freetext(0,13){73}
\freetext(0,12){72}
\freetext(0,11){71}
\freetext(0,10){70}
\freetext(0,9){69}
\freetext(0,8){68}
\freetext(0,7){67}
\freetext(0,6){66}
\freetext(0,5){65}
\freetext(0,4){64}
\freetext(0,3){63}
\freetext(0,2){62}
\freetext(0,1){61}
\freetext(0,0){60}

\end{graph}
\caption{Hasse diagram of nilpotent orbits in the case of trivectors; top half}\label{fig:3vec1}
\end{figure}

\begin{figure}

\begin{graph}(10,21)

\roundnode{n63}(6,20)\nodetext{n63}{63}
\roundnode{n64}(5,20)\nodetext{n64}{64}

\roundnode{n65}(3,19)\nodetext{n65}{65}
\roundnode{n66}(6,19)\nodetext{n66}{66}
\roundnode{n67}(4,19)\nodetext{n67}{67}
\roundnode{n68}(7,19)\nodetext{n68}{68}

\roundnode{n69}(5,18)\nodetext{n69}{69}
\roundnode{n70}(7,18)\nodetext{n70}{70}

\roundnode{n71}(4,17)\nodetext{n71}{71}

\roundnode{n72}(5,16)\nodetext{n72}{72}
\roundnode{n73}(6,16)\nodetext{n73}{73}
\roundnode{n74}(7,16)\nodetext{n74}{74}

\roundnode{n75}(4,15)\nodetext{n75}{75}

\roundnode{n76}(6,14)\nodetext{n76}{76}
\roundnode{n77}(3,14)\nodetext{n77}{77}
\roundnode{n78}(7,14)\nodetext{n78}{78}
\roundnode{n79}(4,14)\nodetext{n79}{79}

\roundnode{n80}(7,13)\nodetext{n80}{80}
\roundnode{n81}(5,13)\nodetext{n81}{81}
\roundnode{n82}(6,13)\nodetext{n82}{82}

\roundnode{n83}(3,12)\nodetext{n83}{83}
\roundnode{n84}(6,12)\nodetext{n84}{84}

\roundnode{n85}(6,11)\nodetext{n85}{85}

\roundnode{n86}(7,10)\nodetext{n86}{86}
\roundnode{n87}(5,10)\nodetext{n87}{87}

\roundnode{n88}(5,9)\nodetext{n88}{88}

\roundnode{n89}(8.8,8)\nodetext{n89}{89}
\roundnode{n90}(6,8)\nodetext{n90}{90}
\roundnode{n91}(1.5,8)\nodetext{n91}{91}

\roundnode{n92}(5,7)\nodetext{n92}{92}
\roundnode{n93}(6,7)\nodetext{n93}{93}

\roundnode{n94}(5.5,6)\nodetext{n94}{94}

\roundnode{n95}(5.5,5.1)\nodetext{n95}{95}

\roundnode{n96}(4.5,4.2)\nodetext{n96}{96}

\roundnode{n97}(5.5,3.3)\nodetext{n97}{97}

\roundnode{n98}(8,2.4)\nodetext{n98}{98}

\roundnode{n99}(6.5,1.5)\nodetext{n99}{99}

\roundnode{n100}(5.5,0.8)\nodetext{n100}{100}

\roundnode{n101}(5.5,0)\nodetext{n101}{101}

\edge{n64}{n66}
\edge{n63}{n66}
\edge{n63}{n68}
\edge{n67}{n69}
\edge{n66}{n69}
\edge{n66}{n70}
\edge{n68}{n70}

\edge{n65}{n71}
\edge{n64}{n71}
\edge{n71}{n72}
\edge{n69}{n72}
\edge{n69}{n73}
\edge{n69}{n74}
\edge{n70}{n74}
\edge{n71}{n75}
\edge{n66}{n75}

\edge{n67}{n77}
\edge{n75}{n79}
\edge{n70}{n79}
\edge{n75}{n76}
\edge{n72}{n76}
\edge{n73}{n78}
\edge{n74}{n78}

\edge{n77}{n81}
\edge{n73}{n81}
\edge{n77}{n82}
\edge{n79}{n82}
\edge{n76}{n82}
\edge{n74}{n82}
\edge{n73}{n80}
\edge{n76}{n80}
\edge{n78}{n84}
\edge{n81}{n84}
\edge{n82}{n84}
\edge{n80}{n84}
\edge{n73}{n83}

\edge{n83}{n85}
\edge{n81}{n85}
\edge{n80}{n85}
\edge{n76}{n86}
\edge{n85}{n87}
\edge{n84}{n87}
\edge{n87}{n88}
\edge{n78}{n89}
\edge{n86}{n90}
\edge{n80}{n90}
\edge{n86}{n91}
\edge{n82}{n91}

\edge{n88}{n92}
\edge{n90}{n92}
\edge{n91}{n92}
\edge{n88}{n93}
\edge{n89}{n93}
\edge{n92}{n94}
\edge{n93}{n94}
\edge{n94}{n95}
\edge{n94}{n96}
\edge{n95}{n97}
\edge{n96}{n97}

\edge{n90}{n98}
\edge{n95}{n99}
\edge{n98}{n99}
\edge{n97}{n100}
\edge{n99}{n100}
\edge{n100}{n101}

\freetext(0,20.5){Dimension}
\freetext(0,20){62}
\freetext(0,19){61}
\freetext(0,18){60}
\freetext(0,17){59}
\freetext(0,16){58}
\freetext(0,15){57}
\freetext(0,14){56}
\freetext(0,13){55}
\freetext(0,12){54}
\freetext(0,11){53}
\freetext(0,10){52}
\freetext(0,9){51}
\freetext(0,8){49}
\freetext(0,7){48}
\freetext(0,6){45}
\freetext(0,5.1){42}
\freetext(0,4.2){38}
\freetext(0,3.3){37}
\freetext(0,2.4){36}
\freetext(0,1.5){35}
\freetext(0,0.8){30}
\freetext(0,0){19}

\end{graph}
\caption{Hasse diagram of nilpotent orbits in the case of trivectors; bottom half}
\label{fig:3vec2}
\end{figure}

\begin{longtable}{|c|c|c|c|}
\caption{Characteristics of the nilpotent orbits in the case of 3-vectors.}\label{tab:3vecchars}
\endfirsthead
\hline
\multicolumn{4}{|l|}{\small\slshape Characteristics.} \\
\hline 
\endhead
\hline 
\endfoot
\endlastfoot

\hline

no. &  characteristic & no. & characterisitc \\
\hline

 1 & 6 6 6 6 6 6 6 12 &  2 & 6 6 6 0 6 6 6 6 \\ 
 3 &  6 6 6 0 6 0 6 6 & 4 & 6 0 6 0 6 6 0 6 \\ 
 5 &  0 6 0 6 0 6 0 6 & 6 & 6 1 5 1 5 6 1 5 \\ 
 7 &  6 0 6 0 6 0 0 6 & 8 & 0 6 0 0 6 0 6 0 \\ 
 9 &  0 0 6 0 0 6 0 6 & 10& 6 1 5 1 5 0 1 5 \\ 
 11&  1 5 1 1 4 1 5 1 & 12& 2 2 2 2 2 4 2 2 \\ 
 13&  0 1 5 0 1 5 1 5 & 14& 2 2 2 2 2 2 2 2 \\ 
 15&  1 1 4 1 1 5 1 4 & 16& 6 0 0 0 6 0 0 6 \\ 
 17&  3 0 3 3 0 6 0 3 & 18& 0 6 0 0 0 6 0 0 \\ 
 19&  2 2 0 2 2 2 2 2 & 20& 6 0 1 0 5 0 1 5 \\ 
 21&  0 3 0 3 0 3 3 0 & 22& 1 4 0 1 1 4 1 1 \\ 
 23&  1 4 1 1 1 3 1 1 & 24& 2 0 4 2 0 6 0 4 \\ 
 25&  0 0 6 0 0 0 0 6 & 26& 6 1 0 1 4 1 0 5 \\ 
 27&  3 0 3 0 3 0 3 0 & 28& 0 4 0 2 0 4 2 0 \\ 
 29&  1 1 2 1 1 1 1 4 & 30& 0 1 5 0 0 1 0 5 \\ 
 31&  2 2 2 2 0 2 0 2 & 32& 1 0 5 0 1 0 1 4 \\ 
 33&  2 0 2 2 0 2 0 4 & 34& 0 0 0 0 6 0 0 0 \\ 
 35&  1 1 4 1 0 1 0 4 & 36& 2 0 2 0 2 0 2 2 \\ 
 37&  0 0 1 0 5 0 0 1 & 38& 2 1 1 1 1 1 1 2 \\ 
 39&  0 1 0 1 4 0 1 0 & 40& 1 0 1 0 4 0 1 1 \\ 
 41&  3 0 3 3 0 0 0 3 & 42& 3 0 0 0 3 0 3 0 \\ 
 43&  1 1 0 1 3 1 0 1 & 44& 0 3 0 0 0 3 0 3 \\ 
 45&  1 0 1 1 2 1 1 1 & 46& 3 0 1 0 2 1 2 0 \\ 
 47&  2 0 4 2 0 0 0 4 & 48& 1 1 1 1 1 1 1 1 \\ 
 49&  2 0 0 0 4 0 2 0 & 50& 0 2 0 0 2 2 0 2 \\ 
 51&  1 1 0 1 1 2 1 1 & 52& 2 0 1 0 3 1 1 0 \\ 
 53&  1 1 1 1 0 1 1 2 & 54& 0 3 0 0 3 0 0 0 \\ 
 55&  0 0 0 3 0 3 0 0 & 56& 2 0 2 0 0 2 0 2 \\ 
 57&  0 0 0 0 0 0 0 6 & 58& 0 4 0 0 2 0 0 0 \\ 
 59&  0 0 3 0 0 0 3 0 & 60& 1 1 1 1 1 0 1 1 \\ 
 61&  0 1 0 2 0 3 0 1 & 62& 0 0 0 0 1 0 0 5 \\ 
 63&  1 1 1 0 1 0 2 1 & 64& 0 1 2 0 1 0 2 0 \\ 
 65&  1 2 1 1 0 0 1 1 & 66& 0 2 0 0 2 0 0 2 \\ 
 67&  1 0 1 1 0 3 1 0 & 68& 0 0 1 0 0 0 1 4 \\ 
 69&  1 1 0 1 1 0 1 1 & 70& 0 1 0 0 1 0 0 4 \\ 
 71&  0 2 2 0 0 0 2 0 & 72& 2 0 1 0 1 1 0 1 \\ 
 73&  1 0 1 1 0 1 1 0 & 74& 1 0 0 1 0 0 1 3 \\ 
 75&  0 3 0 0 0 0 0 3 & 76& 1 2 0 0 0 1 0 2 \\ 
 77&  0 0 2 0 0 4 0 0 & 78& 0 1 0 0 1 0 1 2 \\ 
 79&  0 2 0 0 0 0 0 4 & 80& 1 1 0 1 0 1 0 1 \\ 
 81&  0 0 2 0 0 2 0 0 & 82& 1 1 0 0 0 1 0 3 \\ 
 83&  0 0 0 3 0 0 0 0 & 84& 1 0 0 1 0 1 0 2 \\ 
 85&  0 1 0 2 0 0 0 1 & 86& 2 1 0 1 0 0 0 2 \\ 
 87&  0 0 0 2 0 0 0 2 & 88& 0 1 0 1 0 0 1 1 \\ 
 89&  0 0 0 0 0 0 3 0 & 90& 2 0 0 1 0 1 0 0 \\ 
 91&  2 0 0 1 0 0 0 3 & 92& 1 0 1 0 0 1 0 1 \\ 
 93&  0 0 0 1 0 0 2 0 & 94& 0 1 0 0 0 1 1 0 \\ 
 95&  1 0 0 1 0 0 1 0 & 96& 0 0 0 0 0 2 0 0 \\ 
 97&  0 0 1 0 0 1 0 0 & 98& 3 0 0 0 0 0 0 0 \\ 
 99&  2 0 0 0 0 0 1 0 & 100& 1 0 0 0 1 0 0 0 \\ 
 101&  0 0 1 0 0 0 0 0 & & \\

\hline

\end{longtable}

\begin{longtable}{|c|c|c|c|}
\caption{Characteristics of the nilpotent orbits.}\label{tab:5ordchars}
\endfirsthead
\hline
\multicolumn{4}{|l|}{\small\slshape Characteristics.} \\
\hline 
\endhead
\hline 
\endfoot
\endlastfoot

\hline

no. &  characteristic & no. & characterisitc \\
\hline

1& 10 10 10 10 10 10 10 20 & 2& 10 10 0 10 10 10 10 10 \\ 
3& 10 0 10 10 10 0 10 10 & 4& 0 10 0 10 10 10 0 10 \\ 
5& 10 0 10 0 0 10 0 10 & 6& 1 9 1 9 10 10 1 9 \\ 
7& 10 0 3 7 7 3 7 3 & 8& 3 7 3 7 3 4 3 3 \\ 
9& 0 10 0 0 10 0 0 10 & 10& 0 0 10 0 0 0 10 0 \\ 
11& 1 9 0 1 10 0 1 9 & 12& 2 0 8 2 2 2 6 2 \\ 
13& 0 3 7 0 3 0 7 3 & 14& 0 0 0 0 0 0 0 10 \\ 
15& 1 8 1 1 10 1 1 8 & 16& 3 3 4 3 3 3 1 3 \\ 
17& 5 0 5 0 0 5 5 0 & 18& 1 2 7 1 3 1 6 3 \\ 
19& 0 5 0 5 0 5 0 5 & 20& 0 0 1 0 0 0 1 9 \\ 
21& 2 2 2 4 2 4 2 2 & 22& 6 0 4 0 0 6 4 0 \\ 
23& 3 3 1 3 3 1 2 4 & 24& 0 3 0 7 4 3 0 3 \\ 
25& 1 0 1 0 0 1 1 8 & 26& 0 1 0 1 0 1 0 9 \\ 
27& 0 8 2 0 10 2 0 8 & 28& 4 2 2 2 4 0 2 4 \\ 
29& 0 8 2 0 0 0 2 0 & 30& 2 2 2 2 2 2 2 2 \\ 
31& 1 2 1 6 4 3 1 2 & 32& 1 1 0 1 1 0 1 8 \\ 
33& 1 0 1 1 1 1 1 7 & 34& 3 2 0 5 3 5 0 2 \\ 
35& 3 2 0 5 2 0 3 0 & 36& 1 1 1 1 1 1 1 6 \\ 
37& 2 0 0 0 0 2 0 8 & 38& 0 2 0 0 2 0 2 6 \\ 
39& 0 0 5 0 0 5 0 0 & 40& 3 1 2 1 3 1 2 1 \\ 
41& 2 2 0 6 4 4 0 2 & 42& 1 2 1 3 2 1 1 2 \\ 
43& 2 0 1 0 1 1 0 8 & 44& 3 3 0 4 1 0 3 0 \\ 
45& 0 5 0 0 0 0 0 5 & 46& 0 0 0 3 3 0 0 7 \\ 
47& 1 1 1 1 1 1 2 4 & 48& 1 1 0 1 2 1 1 6 \\ 
49& 0 2 2 2 2 2 0 2 & 50& 0 4 0 0 0 0 0 6 \\ 
51& 2 0 2 0 2 0 2 4 & 52& 0 0 3 0 0 3 0 4 \\ 
53& 1 1 1 2 1 1 1 4 & 54& 1 3 1 1 0 1 1 3 \\ 
55& 0 1 0 2 3 0 1 6 & 56& 0 1 2 1 1 2 0 4 \\ 
57& 1 2 1 1 0 1 1 4 & 58& 0 3 2 0 0 2 0 3 \\ 
59& 1 1 1 1 1 2 1 3 & 60& 1 0 1 1 3 1 0 6 \\ 
61& 0 2 0 2 2 0 2 2 & 62& 0 0 0 0 0 0 5 0 \\ 
63& 2 1 1 2 1 1 1 1 & 64& 0 2 2 0 0 2 0 4 \\ 
65& 1 1 1 1 2 1 1 2 & 66& 2 0 1 2 1 1 1 2 \\ 
67& 0 0 1 0 0 1 4 0 & 68& 0 3 0 0 0 0 3 1 \\ 
69& 5 0 0 0 5 0 0 0 & 70& 0 1 0 1 1 0 4 0 \\ 
71& 1 2 0 1 0 1 2 1 & 72& 0 0 0 5 0 0 0 0 \\ 
73& 2 1 0 1 3 0 1 2 & 74& 1 0 1 0 1 1 3 0 \\ 
75& 3 0 2 0 1 0 2 0 & 76& 1 1 1 1 1 1 1 1 \\ 
77& 0 2 0 0 0 0 4 0 & 78& 0 1 0 4 0 0 0 1 \\ 
79& 0 0 2 0 4 0 0 6 & 80& 2 1 1 0 1 0 2 1 \\ 
81& 1 1 0 1 0 1 3 0 & 82& 0 0 0 4 0 0 0 2 \\ 
83& 1 0 1 1 1 1 2 0 & 84& 0 1 2 0 2 0 1 1 \\ 
85& 0 1 0 3 0 0 1 1 & 86& 2 0 0 0 2 0 3 0 \\ 
87& 2 0 1 1 1 1 0 1 & 88& 0 0 2 0 2 0 2 0 \\ 
89& 0 2 0 0 0 2 1 0 & 90& 2 0 1 0 1 0 3 0 \\ 
91& 1 0 1 2 0 1 0 1 & 92& 0 0 0 3 0 0 2 0 \\ 
93& 1 1 0 1 1 1 1 0 & 94& 0 1 0 2 0 1 1 0 \\ 
95& 0 1 0 0 0 3 0 0 & 96& 1 0 1 1 1 0 1 0 \\ 
97& 1 0 0 1 1 2 0 0 & 98& 3 0 0 0 0 0 0 1 \\ 
99& 0 0 0 2 0 2 0 0 & 100& 2 0 0 1 0 0 1 0 \\ 
101& 0 0 2 0 0 0 1 0 & 102& 0 1 0 1 1 1 0 0 \\ 
103& 1 0 1 0 0 1 0 0 & 104& 0 0 0 1 2 0 0 0 \\ 
105& 0 1 0 0 1 0 0 0 & &\\
\hline
\end{longtable}

\unitlength=1cm
\begin{figure}[htb]

\begin{graph}(10,21)
\roundnode{n1}(4.5,20)\nodetext{n1}{1}
\roundnode{n2}(4,19)\nodetext{n2}{2}
\roundnode{n3}(5,19)\nodetext{n3}{3}
\roundnode{n4}(4,18)\nodetext{n4}{4}
\roundnode{n5}(5,18)\nodetext{n5}{5}
\roundnode{n6}(2.5,17)\nodetext{n6}{6}
\roundnode{n7}(3.5,17)\nodetext{n7}{7}
\roundnode{n8}(6.5,17)\nodetext{n8}{8}
\roundnode{n9}(5.5,17)\nodetext{n9}{9}
\roundnode{n10}(4.5,17)\nodetext{n10}{10}
\roundnode{n11}(3.5,16)\nodetext{n11}{11}
\roundnode{n12}(5.5,16)\nodetext{n12}{12}
\roundnode{n13}(4.5,16)\nodetext{n13}{13}
\roundnode{n14}(6.5,16)\nodetext{n14}{\bf 14}

\roundnode{n15}(2.5,15)\nodetext{n15}{15}
\roundnode{n16}(7.5,15)\nodetext{n16}{16}
\roundnode{n17}(3.5,15)\nodetext{n17}{17}
\roundnode{n18}(4.5,15)\nodetext{n18}{18}
\roundnode{n19}(5.5,15)\nodetext{n19}{19}
\roundnode{n20}(6.5,15)\nodetext{n20}{\bf 20}

\roundnode{n24}(6.5,14)\nodetext{n24}{24}
\roundnode{n21}(2.5,14)\nodetext{n21}{21}
\roundnode{n22}(4.8,14)\nodetext{n22}{22}
\roundnode{n23}(5.5,14)\nodetext{n23}{23}
\roundnode{n25}(3.5,14)\nodetext{n25}{\bf 25}
\roundnode{n26}(7.5,14)\nodetext{n26}{\bf 26}

\roundnode{n27}(1,12.5)\nodetext{n27}{27}
\roundnode{n28}(4.5,12.5)\nodetext{n28}{28}
\roundnode{n29}(8.6,12.5)\nodetext{n29}{29}
\roundnode{n30}(7.6,12.5)\nodetext{n30}{30}
\roundnode{n31}(3.7,12.5)\nodetext{n31}{31}
\roundnode{n32}(5.2,12.5)\nodetext{n32}{\bf 32}
\roundnode{n33}(1.5,12.5)\nodetext{n33}{\bf 33}

\roundnode{n34}(1.7,11)\nodetext{n34}{34}
\roundnode{n35}(7,11)\nodetext{n35}{35}
\roundnode{n36}(5,11)\nodetext{n36}{\bf 36}
\roundnode{n37}(1,11)\nodetext{n37}{\bf 37}
\roundnode{n38}(0.2,11)\nodetext{n38}{\bf 38}

\roundnode{n39}(5.5,9.2)\nodetext{n39}{39}
\roundnode{n40}(7,9.2)\nodetext{n40}{40}
\roundnode{n41}(3.5,9.2)\nodetext{n41}{41}
\roundnode{n42}(8,9.2)\nodetext{n42}{42}
\roundnode{n43}(2.5,9.2)\nodetext{n43}{\bf 43}
\roundnode{n44}(9,9.2)\nodetext{n44}{44}
\roundnode{n45}(10,9.2)\nodetext{n45}{45}
\roundnode{n46}(1.3,9.2)\nodetext{n46}{\bf 46}
\roundnode{n47}(0.5,9.2)\nodetext{n47}{\bf 47}
\roundnode{n48}(4.5,9.2)\nodetext{n48}{\bf 48}

\roundnode{n49}(4,6.5)\nodetext{n49}{49}
\roundnode{n50}(7.7,6.5)\nodetext{n50}{\bf 50}
\roundnode{n51}(3.3,6.5)\nodetext{n51}{\bf 51}
\roundnode{n52}(1.5,6.5)\nodetext{n52}{\bf 52}
\roundnode{n53}(5,6.5)\nodetext{n53}{\bf 53}
\roundnode{n54}(6,6.5)\nodetext{n54}{54}
\roundnode{n55}(2.5,6.5)\nodetext{n55}{\bf 55}

\roundnode{n56}(3,5.3)\nodetext{n56}{\bf 56}
\roundnode{n57}(5,5.3)\nodetext{n57}{\bf 57}
\roundnode{n58}(4,5.3)\nodetext{n58}{58}
\roundnode{n59}(2,5.3)\nodetext{n59}{\bf 59}

\roundnode{n60}(2.2,4.3)\nodetext{n60}{\bf 60}
\roundnode{n61}(3.2,4.3)\nodetext{n61}{\bf 61}
\roundnode{n62}(1.3,4.3)\nodetext{n62}{\it{\color{blue}62}}
\roundnode{n63}(5,4.3)\nodetext{n63}{63}
\roundnode{n64}(4,4.3)\nodetext{n64}{\bf 64}

\roundnode{n65}(2.2,3.3)\nodetext{n65}{\bf 65}
\roundnode{n66}(4,3.3)\nodetext{n66}{\bf 66}
\roundnode{n67}(1.4,3.3)\nodetext{n67}{\it{\color{blue}67}}
\roundnode{n68}(3.2,3.3)\nodetext{n68}{\bf 68}

\roundnode{n69}(5.3,2)\nodetext{n69}{69}
\roundnode{n70}(2.2,2)\nodetext{n70}{\it{\color{blue}70}}
\roundnode{n71}(3.2,2)\nodetext{n71}{\bf 71}
\roundnode{n72}(8.2,2)\nodetext{n72}{72}
\freetext(8.8,2){\tiny (4,5)}

\roundnode{n73}(4.7,1)\nodetext{n73}{\bf 73}
\roundnode{n74}(2.7,1)\nodetext{n74}{\it{\color{blue}74}}
\roundnode{n75}(6.2,1)\nodetext{n75}{75}
\roundnode{n76}(2,1)\nodetext{n76}{\bf 76}
\roundnode{n77}(3.9,1)\nodetext{n77}{\it{\color{blue}77}}
\roundnode{n78}(7.6,1)\nodetext{n78}{78}
\freetext(8.2,1){\tiny (4,5)}

\roundnode{n79}(0.6,0)\nodetext{n79}{\bf 79}
\roundnode{n80}(4,0)\nodetext{n80}{\bf 80}
\roundnode{n81}(3.3,0)\nodetext{n81}{\it{\color{blue}81}}
\roundnode{n82}(5.5,0)\nodetext{n82}{\underline{82}}

\edge{n1}{n2}
\edge{n1}{n3}
\edge{n3}{n4}
\edge{n2}{n4}
\edge{n3}{n5}
\edge{n2}{n5}
\edge{n4}{n6}
\edge{n5}{n7}
\edge{n4}{n7}
\edge{n5}{n8}
\edge{n5}{n9}
\edge{n4}{n9}
\edge{n5}{n10}
\edge{n4}{n10}
\edge{n9}{n11}
\edge{n7}{n11}
\edge{n6}{n11}
\edge{n10}{n12}
\edge{n8}{n12}
\edge{n7}{n12}
\edge{n10}{n13}
\edge{n9}{n13}
\edge{n7}{n13}
\edge{n9}{n14}
\edge{n11}{n15}
\edge{n12}{n16}
\edge{n13}{n17}
\edge{n11}{n17}
\edge{n13}{n18}
\edge{n12}{n18}
\edge{n11}{n18}
\edge{n13}{n19}
\edge{n12}{n19}
\edge{n14}{n20}
\edge{n13}{n20}
\edge{n11}{n20}
\edge{n19}{n21}
\edge{n17}{n21}
\edge{n15}{n21}
\edge{n18}{n21}
\edge{n17}{n22}
\edge{n19}{n23}
\edge{n18}{n23}
\edge{n17}{n23}
\edge{n16}{n23}
\edge{n19}{n24}
\edge{n20}{n25}
\edge{n17}{n25}
\edge{n15}{n25}
\edge{n20}{n26}
\edge{n19}{n26}
\edge{n18}{n26}
\edge{n15}{n27}
\edge{n23}{n28}
\edge{n22}{n28}
\edge{n16}{n29}
\edge{n23}{n30}
\edge{n21}{n30}
\edge{n24}{n31}
\edge{n21}{n31}
\edge{n26}{n32}
\edge{n25}{n32}
\edge{n23}{n32}
\edge{n26}{n33}
\edge{n25}{n33}
\edge{n21}{n33}
\edge{n22}{n34}
\edge{n21}{n34}
\edge{n24}{n35}
\edge{n23}{n35}
\edge{n33}{n36}
\edge{n32}{n36}
\edge{n30}{n36}
\edge{n27}{n37}
\edge{n25}{n37}
\edge{n22}{n37}
\edge{n33}{n38}
\edge{n30}{n39}
\edge{n27}{n39}
\edge{n34}{n40}
\edge{n30}{n40}
\edge{n28}{n40}
\edge{n34}{n41}
\edge{n31}{n41}
\edge{n27}{n41}
\edge{n35}{n42}
\edge{n31}{n42}
\edge{n30}{n42}
\edge{n37}{n43}
\edge{n32}{n43}
\edge{n28}{n43}
\edge{n35}{n44}
\edge{n29}{n44}
\edge{n28}{n44}
\edge{n30}{n45}
\edge{n29}{n45}
\edge{n33}{n46}
\edge{n31}{n46}
\edge{n38}{n47}
\edge{n36}{n47}
\edge{n38}{n48}
\edge{n37}{n48}
\edge{n34}{n48}
\edge{n42}{n49}
\edge{n41}{n49}
\edge{n40}{n49}
\edge{n39}{n49}
\edge{n45}{n50}
\edge{n36}{n50}
\edge{n48}{n51}
\edge{n47}{n51}
\edge{n43}{n51}
\edge{n40}{n51}
\edge{n48}{n52}
\edge{n47}{n52}
\edge{n39}{n52}
\edge{n47}{n53}
\edge{n46}{n53}
\edge{n42}{n53}
\edge{n45}{n54}
\edge{n44}{n54}
\edge{n42}{n54}
\edge{n40}{n54}
\edge{n48}{n55}
\edge{n46}{n55}
\edge{n41}{n55}

\edge{n55}{n56}
\edge{n53}{n56}
\edge{n52}{n56}
\edge{n51}{n56}
\edge{n49}{n56}
\edge{n54}{n57}
\edge{n53}{n57}
\edge{n51}{n57}
\edge{n50}{n57}
\edge{n54}{n58}
\edge{n49}{n58}
\edge{n52}{n59}
\edge{n51}{n59}
\edge{n55}{n60}
\edge{n59}{n61}
\edge{n56}{n61}
\edge{n47}{n62}
\edge{n58}{n63}
\edge{n58}{n64}
\edge{n57}{n64}
\edge{n56}{n64}
\edge{n61}{n65}
\edge{n60}{n65}
\edge{n64}{n66}
\edge{n63}{n66}
\edge{n61}{n66}
\edge{n62}{n67}
\edge{n59}{n67}
\edge{n64}{n68}
\edge{n61}{n68}
\edge{n40}{n69}
\edge{n67}{n70}
\edge{n61}{n70}
\edge{n68}{n71}
\edge{n66}{n71}
\edge{n65}{n71}
\edge{n42}{n72}

\edge{n69}{n73}
\edge{n65}{n73}
\edge{n70}{n74}
\edge{n65}{n74}
\edge{n69}{n75}
\edge{n63}{n75}
\edge{n71}{n76}
\edge{n70}{n77}
\edge{n68}{n77}
\edge{n72}{n78}
\edge{n63}{n78}
\edge{n60}{n79}
\edge{n75}{n80}
\edge{n73}{n80}
\edge{n71}{n80}
\edge{n77}{n81}
\edge{n74}{n81}
\edge{n71}{n81}
\edge{n78}{n82}
\edge{n66}{n82}

\freetext(0,20.5){Dimension}
\freetext(-0.5,20){48}
\freetext(-0.5,19){47}
\freetext(-0.5,18){46}
\freetext(-0.5,17){45}
\freetext(-0.5,16){44}
\freetext(-0.5,15){43}
\freetext(-0.5,14){42}
\freetext(-0.5,12.5){41}
\freetext(-0.5,11){40}
\freetext(-0.5,9.2){39}
\freetext(-0.5,6.5){38}
\freetext(-0.5,5.3){37}
\freetext(-0.5,4.3){36}
\freetext(-0.5,3.3){35}
\freetext(-0.5,2){34}
\freetext(-0.5,1){33}
\freetext(-0.5,0){32}

\end{graph}
\caption{Hasse diagram of nilpotent orbits; top}\label{fig:5ord1}
\end{figure}

\vskip1ex

\begin{figure}[htb]

\begin{graph}(10,15)
\roundnode{n73}(4.7,14)\nodetext{n73}{\bf 73}
\roundnode{n74}(2.7,14)\nodetext{n74}{\it{\color{blue}74}}
\roundnode{n75}(6.2,14)\nodetext{n75}{75}
\roundnode{n76}(2,14)\nodetext{n76}{\bf 76}
\roundnode{n77}(3.9,14)\nodetext{n77}{\it{\color{blue}77}}
\roundnode{n78}(7.6,14)\nodetext{n78}{78}
\freetext(8.2,14){\tiny (4,5)}

\roundnode{n79}(0.6,13)\nodetext{n79}{\bf 79}
\roundnode{n80}(4.3,13)\nodetext{n80}{\bf 80}
\roundnode{n81}(3,13)\nodetext{n81}{\it{\color{blue}81}}
\roundnode{n82}(5.5,13)\nodetext{n82}{\underline{82}}

\roundnode{n83}(3.3,12)\nodetext{n83}{\it{\color{blue}83}}
\roundnode{n84}(1.5,12)\nodetext{n84}{\bf 84}
\roundnode{n85}(5.3,12)\nodetext{n85}{\underline{85}}

\roundnode{n86}(1.4,11)\nodetext{n86}{\it{\color{blue}86}}
\roundnode{n87}(3.2,11)\nodetext{n87}{\bf 87}

\roundnode{n88}(2.2,10)\nodetext{n88}{\it{\color{blue}88}}
\roundnode{n89}(5.8,10)\nodetext{n89}{\it{\color{blue}89}}
\roundnode{n90}(4.9,10)\nodetext{n90}{\it{\color{blue}90}}
\roundnode{n91}(3.2,10)\nodetext{n91}{\underline{91}}

\roundnode{n92}(5.2,9)\nodetext{n92}{\tt 92}
\roundnode{n93}(4.5,9)\nodetext{n93}{\it{\color{blue}93}}

\roundnode{n94}(4,8)\nodetext{n94}{\tt 94}

\roundnode{n95}(5.8,7)\nodetext{n95}{95}
\freetext(6.5,7){\tiny (5,2)}

\roundnode{n96}(4,6)\nodetext{n96}{\tt 96}

\roundnode{n97}(5.8,5)\nodetext{n97}{97}
\freetext(6.5,5){\tiny (5,2)}
\roundnode{n98}(1.5,5)\nodetext{n98}{\bf 98}

\roundnode{n99}(4.8,4)\nodetext{n99}{99}
\freetext(5.5,4){\tiny (4,2)}
\roundnode{n100}(3,4)\nodetext{n100}{\tt 100}

\roundnode{n101}(1.4,3)\nodetext{n101}{101}
\freetext(2.1,3){\tiny (3,3)}
\roundnode{n102}(4,3)\nodetext{n102}{102}
\freetext(4.7,3){\tiny (4,2)}

\roundnode{n103}(2.5,2)\nodetext{n103}{103}
\freetext(3.2,2){\tiny (3,2)}

\roundnode{n104}(4,1)\nodetext{n104}{104}
\freetext(4.7,1){\tiny (4,1)}

\roundnode{n105}(3.5,0)\nodetext{n105}{105}
\freetext(4.2,0){\tiny (2,1)}

\edge{n75}{n80}
\edge{n73}{n80}
\edge{n77}{n81}
\edge{n74}{n81}
\edge{n78}{n82}
\edge{n81}{n83}
\edge{n76}{n83}

\edge{n79}{n84}
\edge{n76}{n84}
\edge{n82}{n85}
\edge{n76}{n85}
\edge{n79}{n86}
\edge{n74}{n86}
\edge{n73}{n86}
\edge{n80}{n87}
\edge{n76}{n87}
\edge{n86}{n88}
\edge{n84}{n88}
\edge{n83}{n88}
\edge{n83}{n89}
\edge{n86}{n90}
\edge{n81}{n90}
\edge{n80}{n90}
\edge{n87}{n91}
\edge{n85}{n91}
\edge{n84}{n91}

\edge{n85}{n92}
\edge{n83}{n92}
\edge{n90}{n93}
\edge{n89}{n93}
\edge{n88}{n93}
\edge{n87}{n93}
\edge{n93}{n94}
\edge{n92}{n94}
\edge{n91}{n94}
\edge{n89}{n95}
\edge{n94}{n96}
\edge{n95}{n97}
\edge{n93}{n97}
\edge{n87}{n98}

\edge{n97}{n99}
\edge{n94}{n99}
\edge{n98}{n100}
\edge{n96}{n100}
\edge{n96}{n101}
\edge{n99}{n102}
\edge{n96}{n102}
\edge{n102}{n103}
\edge{n101}{n103}
\edge{n100}{n103}
\edge{n102}{n104}
\edge{n104}{n105}
\edge{n103}{n105}

\freetext(0,14.5){Dimension}
\freetext(-0.5,14){33}
\freetext(-0.5,13){32}
\freetext(-0.5,12){31}
\freetext(-0.5,11){30}
\freetext(-0.5,10){29}
\freetext(-0.5,9){28}
\freetext(-0.5,8){27}
\freetext(-0.5,7){26}
\freetext(-0.5,6){25}
\freetext(-0.5,5){24}
\freetext(-0.5,4){22}
\freetext(-0.5,3){21}
\freetext(-0.5,2){18}
\freetext(-0.5,1){14}
\freetext(-0.5,0){11}

\end{graph}
\caption{Hasse diagram of nilpotent orbits; bottom}\label{fig:5ord2}
\end{figure}

\begin{longtable}{|c|c|c|c|}
\caption{Characteristics of the nilpotent orbits.}\label{tab:charsE7}
\endfirsthead
\hline
\multicolumn{4}{|l|}{\small\slshape Characteristics.} \\
\hline 
\endhead
\hline 
\endfoot
\endlastfoot

\hline

no. &  characteristic & no. & characterisitc \\
\hline
1 & 6 6 6 6 6 6 12 & 2 &  6 6 6 0 6 6 6 \\ 
3 & 0 6 0 6 6 6 6 & 4 &  0 6 0 6 0 6 12 \\ 
5 & 2 4 2 6 4 6 6 & 6 &  6 0 6 0 6 0 6 \\ 
7 & 6 1 5 1 5 1 5 & 8 &  0 6 0 6 0 0 6 \\ 
9 & 0 6 0 0 6 6 0 & 10 &  4 0 2 4 2 6 6 \\ 
11 & 0 0 6 0 0 0 6 & 12 &  1 5 0 1 5 6 1 \\ 
13 & 1 5 1 1 4 5 1 & 14 &  2 2 2 2 4 4 2 \\ 
15 & 2 0 4 0 2 0 6 & 16 &  2 2 2 2 2 2 2 \\ 
17 & 0 1 5 0 1 1 5 & 18 &  2 4 2 0 4 6 0 \\ 
19 & 0 6 0 0 0 6 6 & 20 &  2 0 2 2 2 2 4 \\ 
21 & 1 1 4 1 1 1 4 & 22 &  2 1 3 1 1 1 5 \\ 
23 & 0 0 0 6 0 6 0 & 24 &  0 3 0 3 0 0 6 \\ 
25 & 2 2 0 2 2 2 2 & 26 &  1 0 2 3 1 3 3 \\ 
27 & 0 0 0 0 6 0 0 & 28 &  3 0 3 3 0 0 3 \\ 
29 & 0 6 0 0 0 0 0 & 30 &  0 1 0 5 0 6 1 \\ 
31 & 1 2 1 0 2 3 3 & 32 &  1 2 0 3 1 1 5 \\ 
33 & 2 0 4 2 0 0 4 & 34 &  1 4 0 1 1 1 1 \\ 
35 & 0 0 1 0 5 0 1 & 36 &  0 3 0 3 0 3 0 \\ 
37 & 2 0 2 0 2 2 2 & 38 &  0 4 0 2 0 2 0 \\ 
39 & 0 1 0 1 4 1 0 & 40 &  2 2 0 0 2 4 2 \\ 
41 & 0 0 0 0 0 0 6 & 42 &  0 2 0 0 4 0 0 \\ 
43 & 2 1 1 1 1 1 2 & 44 &  0 2 0 4 0 2 4 \\ 
45 & 1 0 1 0 4 1 1 & 46 &  0 2 0 2 0 2 2 \\ 
47 & 1 1 0 1 3 0 1 & 48 &  0 0 0 1 0 1 5 \\ 
49 & 1 0 1 1 2 1 1 & 50 &  3 0 0 0 3 3 0 \\ 
51 & 0 1 0 1 0 1 4 & 52 &  2 0 0 0 4 2 0 \\ 
53 & 1 0 2 0 1 3 0 & 54 &  0 3 0 0 0 0 3 \\ 
55 & 3 0 1 0 2 2 0 & 56 &  0 2 0 0 2 0 2 \\ 
57 & 2 0 1 0 3 1 0 & 58 &  0 2 0 0 0 0 4 \\ 
59 & 0 0 2 0 2 2 0 & 60 &  0 1 0 0 2 0 3 \\ 
61 & 0 1 0 0 0 2 4 & 62 &  1 1 0 1 1 1 1 \\ 
63 & 0 1 0 1 0 2 2 & 64 &  0 0 0 3 0 0 0 \\ 
65 & 1 0 0 1 1 1 2 & 66 &  0 1 0 2 0 0 1 \\ 
67 & 1 0 1 1 0 1 0 & 68 &  0 0 0 2 0 0 2 \\ 
69 & 0 1 0 1 0 1 1 & 70 &  2 0 0 1 0 0 0 \\ 
71 & 0 0 2 0 0 0 0 & 72 &  1 0 1 0 0 0 1 \\ 
73 & 0 0 0 0 0 3 0 & 74 &  0 0 0 1 0 2 0 \\ 
75 & 0 1 0 0 0 1 0 &  & \\
\hline
\end{longtable}

\unitlength=1cm
\begin{figure}[htb]

\begin{graph}(10,21)
\roundnode{n1}(4.5,20)\nodetext{n1}{1}
\roundnode{n2}(4,19)\nodetext{n2}{2}
\roundnode{n3}(5,19)\nodetext{n3}{3}
\roundnode{n4}(2.5,18)\nodetext{n4}{4}
\roundnode{n5}(4,18)\nodetext{n5}{5}
\roundnode{n6}(5,18)\nodetext{n6}{6}

\roundnode{n7}(6,17)\nodetext{n7}{7}
\roundnode{n8}(5,17)\nodetext{n8}{8}
\roundnode{n9}(4,17)\nodetext{n9}{9}
\roundnode{n10}(2.5,17)\nodetext{n10}{10}

\roundnode{n11}(4,16)\nodetext{n11}{11}
\roundnode{n12}(3,16)\nodetext{n12}{12}
\roundnode{n13}(5,16)\nodetext{n13}{13}
\roundnode{n14}(6,16)\nodetext{n14}{14}

\roundnode{n15}(3,15)\nodetext{n15}{15}
\roundnode{n16}(4.6,15)\nodetext{n16}{16}
\roundnode{n17}(4,15)\nodetext{n17}{17}
\roundnode{n18}(2.2,15)\nodetext{n18}{18}

\roundnode{n19}(1.5,14)\nodetext{n19}{19}
\roundnode{n20}(4.5,14)\nodetext{n20}{20}
\roundnode{n21}(5.2,14)\nodetext{n21}{21}
\roundnode{n22}(3,14)\nodetext{n22}{22}
\roundnode{n23}(8,14)\nodetext{n23}{23}

\roundnode{n24}(2.2,13)\nodetext{n24}{24}
\roundnode{n25}(4.7,13)\nodetext{n25}{25}
\roundnode{n26}(5.6,13)\nodetext{n26}{26}
\roundnode{n27}(6.9,13)\nodetext{n27}{\it{\color{blue}27}}
\roundnode{n28}(7.6,13)\nodetext{n28}{28}
\roundnode{n29}(8.2,13)\nodetext{n29}{29}
\roundnode{n30}(3.5,13)\nodetext{n30}{30}

\roundnode{n31}(3,12)\nodetext{n31}{31}
\roundnode{n32}(1.2,12)\nodetext{n32}{32}
\roundnode{n33}(6.8,12)\nodetext{n33}{33}
\roundnode{n34}(7.8,12)\nodetext{n34}{34}
\roundnode{n35}(5,12)\nodetext{n35}{\it{\color{blue}35}}
\roundnode{n36}(4,12)\nodetext{n36}{36}

\roundnode{n37}(3,11)\nodetext{n37}{37}
\roundnode{n38}(7,11)\nodetext{n38}{38}
\roundnode{n39}(5,11)\nodetext{n39}{\it{\color{blue}39}}

\roundnode{n40}(1.3,10)\nodetext{n40}{40}
\roundnode{n41}(0.5,10)\nodetext{n41}{\bf 41}
\roundnode{n42}(6,10)\nodetext{n42}{\it{\color{blue}42}}
\roundnode{n43}(4.8,10)\nodetext{n43}{43}
\roundnode{n44}(2.5,10)\nodetext{n44}{44}
\roundnode{n45}(4,10)\nodetext{n45}{\it{\color{blue}45}}

\roundnode{n46}(3.7,9)\nodetext{n46}{46}
\roundnode{n47}(4.8,9)\nodetext{n47}{\it{\color{blue}47}}
\roundnode{n48}(1.5,9)\nodetext{n48}{\bf 48}

\roundnode{n49}(4.8,8)\nodetext{n49}{\it{\color{blue}49}}
\roundnode{n50}(3,8)\nodetext{n50}{50}
\roundnode{n51}(1.9,8)\nodetext{n51}{\bf 51}

\roundnode{n52}(4.6,7)\nodetext{n52}{\it{\color{blue}52}}
\roundnode{n53}(3.8,7)\nodetext{n53}{53}
\roundnode{n54}(7,7)\nodetext{n54}{54}
\roundnode{n55}(3,7)\nodetext{n55}{55}

\roundnode{n56}(7.2,6)\nodetext{n56}{\it{\color{blue}56}}
\roundnode{n57}(6.2,6)\nodetext{n57}{\it{\color{blue}57}}
\roundnode{n58}(2.7,6)\nodetext{n58}{\bf 58}
\roundnode{n59}(5.2,6)\nodetext{n59}{\it{\color{blue}59}}

\roundnode{n60}(4,5)\nodetext{n60}{60}
\freetext(4.6,4.9){\tiny (5,2)}
\roundnode{n61}(0.7,5)\nodetext{n61}{\bf 61}
\roundnode{n62}(6.8,5)\nodetext{n62}{\it{\color{blue}62}}

\roundnode{n63}(2.7,4)\nodetext{n63}{\bf 63}
\roundnode{n64}(6.5,4)\nodetext{n64}{\tt 64}

\roundnode{n65}(4.4,3)\nodetext{n65}{65}
\freetext(5,3){\tiny (5,2)}
\roundnode{n66}(7,3)\nodetext{n66}{\tt 66}

\roundnode{n67}(7,2)\nodetext{n67}{\tt 67}

\roundnode{n68}(3.5,1)\nodetext{n68}{$\overline{68}$}

\roundnode{n69}(4.2,0)\nodetext{n69}{$\overline{69}$}

\edge{n1}{n2}
\edge{n1}{n3}
\edge{n1}{n4}
\edge{n3}{n5}
\edge{n2}{n5}
\edge{n3}{n6}
\edge{n2}{n6}

\edge{n6}{n7}
\edge{n6}{n8}
\edge{n5}{n8}
\edge{n6}{n9}
\edge{n5}{n9}
\edge{n5}{n10}
\edge{n4}{n10}

\edge{n10}{n11}
\edge{n8}{n11}
\edge{n10}{n12}
\edge{n9}{n12}
\edge{n7}{n13}
\edge{n8}{n13}
\edge{n9}{n13}
\edge{n9}{n14}
\edge{n8}{n14}

\edge{n11}{n15}
\edge{n13}{n16}
\edge{n14}{n16}
\edge{n11}{n17}
\edge{n12}{n17}
\edge{n13}{n17}
\edge{n14}{n17}
\edge{n12}{n18}
\edge{n13}{n18}

\edge{n10}{n19}
\edge{n17}{n20}
\edge{n15}{n20}
\edge{n16}{n21}
\edge{n17}{n21}
\edge{n15}{n22}
\edge{n17}{n22}
\edge{n18}{n22}
\edge{n12}{n23}
\edge{n14}{n23}

\edge{n15}{n24}
\edge{n19}{n24}
\edge{n20}{n25}
\edge{n21}{n25}
\edge{n22}{n25}
\edge{n18}{n26}
\edge{n20}{n26}
\edge{n23}{n26}

\edge{n16}{n27}
\edge{n16}{n28}
\edge{n16}{n29}
\edge{n19}{n30}
\edge{n23}{n30}

\edge{n22}{n31}
\edge{n24}{n31}
\edge{n26}{n31}
\edge{n30}{n31}
\edge{n20}{n32}
\edge{n24}{n32}
\edge{n21}{n33}
\edge{n28}{n33}
\edge{n25}{n34}
\edge{n28}{n34}
\edge{n29}{n34}
\edge{n25}{n35}
\edge{n27}{n35}
\edge{n25}{n36}
\edge{n26}{n36}

\edge{n31}{n37}
\edge{n32}{n37}
\edge{n36}{n37}
\edge{n33}{n38}
\edge{n34}{n38}
\edge{n36}{n38}
\edge{n33}{n39}
\edge{n35}{n39}
\edge{n36}{n39}

\edge{n31}{n40}
\edge{n24}{n41}
\edge{n38}{n42}
\edge{n39}{n42}
\edge{n37}{n43}
\edge{n38}{n43}
\edge{n26}{n44}
\edge{n32}{n44}
\edge{n37}{n45}
\edge{n39}{n45}

\edge{n37}{n46}
\edge{n44}{n46}
\edge{n42}{n47}
\edge{n43}{n47}
\edge{n44}{n47}
\edge{n45}{n47}
\edge{n31}{n48}
\edge{n41}{n48}
\edge{n44}{n48}

\edge{n46}{n49}
\edge{n47}{n49}
\edge{n37}{n50}
\edge{n40}{n50}
\edge{n46}{n51}
\edge{n48}{n51}

\edge{n45}{n52}
\edge{n50}{n52}
\edge{n46}{n53}
\edge{n50}{n53}
\edge{n43}{n54}
\edge{n46}{n54}
\edge{n43}{n55}
\edge{n50}{n55}

\edge{n49}{n56}
\edge{n54}{n56}
\edge{n47}{n57}
\edge{n52}{n57}
\edge{n55}{n57}
\edge{n51}{n58}
\edge{n54}{n58}
\edge{n49}{n59}
\edge{n52}{n59}
\edge{n53}{n59}

\edge{n56}{n60}
\edge{n58}{n60}
\edge{n40}{n61}
\edge{n48}{n61}
\edge{n56}{n62}
\edge{n57}{n62}
\edge{n59}{n62}

\edge{n53}{n63}
\edge{n51}{n63}
\edge{n61}{n63}
\edge{n49}{n64}

\edge{n60}{n65}
\edge{n62}{n65}
\edge{n63}{n65}
\edge{n62}{n66}
\edge{n64}{n66}

\edge{n66}{n67}
\edge{n65}{n68}
\edge{n66}{n68}
\edge{n67}{n69}
\edge{n68}{n69}

\freetext(0,20.5){Dimension}
\freetext(-0.5,20){42}
\freetext(-0.5,19){41}
\freetext(-0.5,18){40}
\freetext(-0.5,17){39}
\freetext(-0.5,16){38}
\freetext(-0.5,15){37}
\freetext(-0.5,14){36}
\freetext(-0.5,13){35}
\freetext(-0.5,12){34}
\freetext(-0.5,11){33}
\freetext(-0.5,10){32}
\freetext(-0.5,9){31}
\freetext(-0.5,8){30}
\freetext(-0.5,7){29}
\freetext(-0.5,6){28}
\freetext(-0.5,5){27}
\freetext(-0.5,4){26}
\freetext(-0.5,3){25}
\freetext(-0.5,2){23}
\freetext(-0.5,1){22}
\freetext(-0.5,0){21}
\end{graph}
\caption{Hasse diagram of nilpotent orbits, $E_7$, top}\label{fig:E7_3}
\end{figure}

\begin{figure}[htb]

\begin{graph}(10,13)
\roundnode{n52}(4.6,12)\nodetext{n52}{\it{\color{blue}52}}
\roundnode{n53}(3.8,12)\nodetext{n53}{53}
\roundnode{n54}(7,12)\nodetext{n54}{54}
\roundnode{n55}(3,12)\nodetext{n55}{55}

\roundnode{n56}(7.2,11)\nodetext{n56}{\it{\color{blue}56}}
\roundnode{n57}(6.2,11)\nodetext{n57}{\it{\color{blue}57}}
\roundnode{n58}(2.7,11)\nodetext{n58}{\bf 58}
\roundnode{n59}(5.2,11)\nodetext{n59}{\it{\color{blue}59}}

\roundnode{n60}(4,10)\nodetext{n60}{60}
\freetext(4.6,9.9){\tiny (5,2)}
\roundnode{n61}(0.7,10)\nodetext{n61}{\bf 61}
\roundnode{n62}(6.8,10)\nodetext{n62}{\it{\color{blue}62}}

\roundnode{n63}(2.7,9)\nodetext{n63}{\bf 63}
\roundnode{n64}(6.5,9)\nodetext{n64}{\tt 64}

\roundnode{n65}(4.4,8)\nodetext{n65}{65}
\freetext(5,8){\tiny (5,2)}
\roundnode{n66}(7,8)\nodetext{n66}{\tt 66}

\roundnode{n67}(7,7)\nodetext{n67}{\tt 67}

\roundnode{n68}(3.5,6)\nodetext{n68}{$\overline{68}$}

\roundnode{n69}(4.2,5)\nodetext{n69}{$\overline{69}$}

\roundnode{n70}(5.2,4)\nodetext{n70}{\tt 70}

\roundnode{n71}(6.2,3)\nodetext{n71}{71}
\freetext(6.8,3){\tiny (3,3)}

\roundnode{n72}(5.2,2)\nodetext{n72}{72}
\freetext(5.8,2){\tiny (3,2)}
\roundnode{n73}(2.2,2)\nodetext{n73}{73}
\freetext(2.8,2){\tiny (6,1)}

\roundnode{n74}(3.2,1)\nodetext{n74}{74}
\freetext(3.8,1){\tiny (4,1)}
\roundnode{n75}(4.5,0)\nodetext{n75}{75}
\freetext(5.2,0){\tiny (2,1)}

\edge{n54}{n56}
\edge{n52}{n57}
\edge{n55}{n57}
\edge{n54}{n58}
\edge{n52}{n59}
\edge{n53}{n59}

\edge{n56}{n60}
\edge{n58}{n60}
\edge{n56}{n62}
\edge{n57}{n62}
\edge{n59}{n62}

\edge{n53}{n63}
\edge{n61}{n63}

\edge{n60}{n65}
\edge{n62}{n65}
\edge{n63}{n65}
\edge{n62}{n66}
\edge{n64}{n66}

\edge{n66}{n67}
\edge{n65}{n68}
\edge{n66}{n68}
\edge{n67}{n69}
\edge{n68}{n69}

\edge{n67}{n70}
\edge{n67}{n71}
\edge{n69}{n72}
\edge{n70}{n72}
\edge{n71}{n72}
\edge{n53}{n73}
\edge{n63}{n73}

\edge{n69}{n74}
\edge{n73}{n74}
\edge{n72}{n75}
\edge{n74}{n75}

\freetext(0,12.5){Dimension}
\freetext(-0.5,12){29}
\freetext(-0.5,11){28}
\freetext(-0.5,10){27}
\freetext(-0.5,9){26}
\freetext(-0.5,8){25}
\freetext(-0.5,7){23}
\freetext(-0.5,6){22}
\freetext(-0.5,5){21}
\freetext(-0.5,4){20}
\freetext(-0.5,3){18}
\freetext(-0.5,2){17}
\freetext(-0.5,1){16}
\freetext(-0.5,0){11}
\end{graph}
\caption{Hasse diagram of nilpotent orbits, $E_7$, bottom}\label{fig:E7_3-2}
\end{figure}

\end{document}